\documentclass{article}

\usepackage{amsmath,amssymb,latexsym,amsthm,mathrsfs,gensymb}
\usepackage{xcolor}

\newtheorem{theo}{Theorem}[section]
\newtheorem{pro}[theo]{Proposition}
\newtheorem{lem}[theo]{Lemma}
\newtheorem{cor}[theo]{Corollary}
 
\newcommand{\ra}{\rightarrow}

\providecommand{\MR}{\relax\ifhmode\unskip\space\fi MR }
\providecommand{\bysame}{\leavevmode\hbox
to3em{\hrulefill}\thinspace}

 \theoremstyle{definition}
\newtheorem{defin}[theo]{Definition}
\newtheorem{exa}[theo]{Example}

\newcommand{\Pp}{\mathbf P}
\newcommand{\Ee}{\mathbf E}
\newcommand{\Z}{\mathbf Z}

\def\<{\langle}
\def\>{\rangle}
\def\weight{\omega}
\def\word{w}

\theoremstyle{remark}
\newtheorem{rem}[theo]{Remark}

\def\eps{\varepsilon}

\def\red{\color{red} }

\title{Long range random walks and associated geometries on groups of polynomial growth}

\author{Zhen-Qing Chen, Takashi Kumagai, Laurent Saloff-Coste,\\ Jian Wang and Tianyi Zheng}
\date{}

\begin{document}

\maketitle

\begin{abstract}  In the context of countable groups of polynomial volume growth, we consider a large class of random walks that
 are allowed to take
long jumps along multiple subgroups according to  power law distributions. For such a random walk, we study the large time behavior
of its probability of return at time $n$ in terms of the key parameters describing the driving measure and the structure of the underlying group. We obtain assorted  estimates including near-diagonal two-sided estimates and the
H\"older continuity  of the solutions of the associated discrete parabolic difference equation. In each case, these estimates involve the construction of a geometry adapted to the walk.
\end{abstract}

\section{Introduction}
\setcounter{equation}{0}

\subsection{Random walks and word-length}
Given a probability measure $\mu$ on a discrete group $G$
with identity element $e$,
a random
walk driven by $\mu$ with initial measure $\nu_0$ is  a $G$-valued
stochastic process $(X_n)_0^\infty$ such that $X_0$ has law $\nu_0$
and  $X_{n+1}=X_n\xi_{n+1}$, where $(\xi_i)_1^\infty$ is a $G$-valued
i.i.d sequence with $\xi_i$ distributed according to $\mu$. This discrete Markov process has transition kernel
$$p(x,y)=\mathbf P( X_{n+1}=y|X_n=x)= \mu(x^{-1}y)$$  and satisfies
$$
\mathbf P(X_{n}=x)=\nu_0*\mu^{(n)} (x),
$$
where $u*v(x)=\sum u(y)v(y^{-1}x)$ and $\mu^{(n)}$ stands for the $n$-fold convolution of $\mu$ with itself.  Understanding the behavior of the function of the discrete time parameter $n$,
$$n\mapsto \mu^{(n)}(e),$$
which represents the return probability to the starting point after $n$ steps,
 is one of the key questions in the study of random walks.  When $\mu$ is symmetric (i.e., $\mu(g^{-1})=\mu(g)$ for all $g\in G$), it is an easy exercise to check that
$$n\mapsto \mu^{(2n)}(e)=\|\mu^{(2n)}\|_\infty=
\max_{x\in G} \mu^{(2n)}(x)
$$
is a non-increasing function of $n$.
The aim of this article is to study, in the context of finitely generated groups of polynomial volume growth, a natural class of random walks that allow for long range jumps. General random walks on countable groups were first considered in Harry
Kesten's
1958 Ph.D. dissertation published as \cite{Kesten}. For further background information, see \cite{KV,LSC-N,Woess}.

The most natural and best studied random walks on a finitely generated group $G$ are driven by  finitely supported  symmetric measures, and it is then natural to assume that the support of the measure generates the group $G$ (otherwise, we can restrict attention to the subgroup generated by the support).  In the study of these random walks,  the word-length distance and associated geometry are very useful. Given a finite symmetric generating set  $S$, the associated word-length of an element $g$ in $G$, $|g|=|g|_{G,S}$, is the least number of generators needed to express $g$ as a product over $S$ in $G$ (by convention, $|e|_S=0$).  The associated (left-invariant) distance
between two elements $x,y\in G$ is $$d(x,y)=d_{G,S}(x,y)=|x^{-1}y|
.$$
The volume growth function of the pair $(G,S)$ is the counting function $$V(r)=\#\{g: |g|\le r\}.$$

We will use the notation $f_1\asymp f_2$  between two real valued functions defined on an abstract domain $D$ (often omitted)  to indicate that
there are constants $c_1,c_2\in (0,\infty)$ such that
$$
\forall\, x\in D,\;\;  c_1f_1(x)\le f_2(x)\le c_2f(x).
$$
We will also use the notation
$ f_1\simeq f_2$ between two positive real functions defined on an appropriate domain $D\subset \mathbb R$
(typically,
$D=[1,\infty)$  or $D=(0,1]$ or also $D=\{0, 1,2,\dots\}$)
to indicate that there are constants
$c_i$, $1\le i\le 4$, such that
$$\forall\, x\in D,\;\; c_1f_1(c_2 x)\le f_2(x)\le c_3f_1(c_4x)$$
(in each case, $c_2x$ and $c_4x$ should be understood appropriately.
Specifically, when $D=[1, \infty)$, $D=(0, 1]$ and $D=\{0, 1,2,\dots\}$,
$c_2x$ and $c_4x$ should be understood
as $(c_2x)\vee 1$ and $(c_4x)\vee 1$, as
$(c_2x)\wedge 1$ and $(c_4x)\wedge 1$, and as
$\lfloor c_2x\rfloor$ and  $\lceil c_4x\rceil$, respectively. Here for $a, b\in \mathbb R$,
$a\vee b:=\max\{a, b\}$,
$a\wedge b:=\min \{a, b\}$, and $\lfloor a\rfloor$ denotes the largest integer not exceeding $a$).

Typically, we assume that at least one of these functions is monotone (otherwise, this notion is not very practical).
Similarly, we define the associated order relations $\preceq$ and $\succeq$ so that
 $f \preceq g$ means that $f(x) \leq c_1 g(c_2 x)$, and so on.
 For instance, when $|\cdot |_1$ and $|\cdot|_2$ are word-length functions associated to two finite symmetric generating sets $S_1,S_2$ of the same group $G$ then, for all $x\in G$,  $|x|_1\asymp |x|_2$.  If $V_1$, $V_2$ are the associated volume growth functions then  $V_1\simeq V_2$. In particular, up to the $\simeq $ equivalence relation,  the volume growth function of a
finitely generated group $G$
does not depend on the choice of the finite generating set $S$, see, e.g., \cite{Gri}.

\begin{defin} A finitely generated group has polynomial volume growth of degree $d$ if,
$V(r)\asymp r^d$ for $r\in [1,\infty)$.
\end{defin}

By a celebrated theorem of M. Gromov,
it suffices that $$\liminf_{r\ra \infty} r^{-A}V(r) <\infty$$
 with some constant $A$ for the group $G$ to have polynomial volume growth of degree $d$ for some integer $d=d(G)\in \{0,1,\dots\}$.  In this context, the tight relation between volume growth and random walk behavior is illustrated by the following result (See also  \cite{Hebisch1993,VSCC,Woess}).
\begin{theo}[N. Varopoulos, {\cite{Varnil}}]
\label{th1}
Let $G$ be a finitely generated group of polynomial volume growth of degree $d$ and let $\mu $ be a finitely supported symmetric probability measure on $G$ with generating support. Then, for all $n\in \{1,2,\dots\}$,
$$\mu^{(2n)}(e)\asymp \frac{1}{V(\sqrt{n})}  \asymp n^{-d/2}.$$\end{theo}
In fact, this result can be generalized in two significant directions by allowing $\mu$ to have finite second moment and by  estimating $\mu^{(2n)}(g)$ for a range of $g$ that depends on $n$.
\begin{theo} \label{th2}
Let $G$ be a finitely generated group of polynomial volume growth of degree $d$ and let $\mu $ be a symmetric probability measure on $G$
with generating support and with  finite second moment, that is,
$\sum_g |g|^2\mu(g)<\infty$.  For simplicity, assume that $\mu(e)>0$. Then, for any fixed $A>0$, we have
$$\forall\, g\in G, \;n\in \{1,2, \dots,\} \mbox{ with }  |g|\le  A \sqrt{n},\;\;
\mu^{(n)}(g)\asymp \frac{1}{V(\sqrt{n})} \asymp n^{-d/2}.
$$
\end{theo}

See, e.g., \cite{Hebisch1993,PSCstab} and the references therein.   This type of estimate is often called a near diagonal estimate. In the result above, the range of order
$\sqrt{n}$ is optimal.  To close this short review and emphasize the importance of the word-length geometry in this context, let us mention briefly two more sophisticated results, namely,  the parabolic Harnack inequality and H\"older continuity for solutions  $(n,x)\mapsto u_n(x)$
of  the parabolic difference equation
\begin{equation} \label{diffeq}
u_{n+1}-u_n=
u_n*(\mu - \delta_e  )
\mbox{ or, equivalently, }  u_{n+1} = u_n*\mu.
\end{equation}
This discrete time evolution equation is parabolic because  it resembles the classical heat equation with the operator
$f\mapsto f*(\mu-\delta_e )$ playing the role of the Laplace operator (note that
$f\mapsto f*(\mu-\delta_e )$ is non-positive definite on $L^2(G)$).
The function
$$(n,x)\mapsto \mu^{(n)}(x)$$ is a global solution of this equation. Note that for equation  (\ref{diffeq}) to make sense and hold in a given subset
$A$, it is necessary that $u_n$ be defined, not only in $A$ but over a set containing $A(\mbox{support}(\mu))^{-1}$.  In the next theorem, $\mu$ is symmetric and has finite support $S$. In such cases, whenever we say that $u_n$ is solution of (\ref{diffeq}) in $[0,T]\times A$, we tacitly assume that $(k,x)\mapsto u_k(x)$ is defined for all $(k,x)\in [0,T]\times
A (S\cup \{e\})$.

\begin{theo}[{\cite{Delmotte1999}(special case)}]\label{T:1.4}
Assume that $G$ has polynomial volume growth and the measure $\mu$ is symmetric, finitely supported with generating support
$S$
containing the identity element, $e$. Then there are constants $C$ and $\alpha>0$ such that the following two properties hold.
\begin{description}
 \item[Parabolic Harnack Inequality]   Any positive solution $u$ of the difference equation {\em (\ref{diffeq})} in the discrete time cylinder $Q=[0,N^2]\times \{x\in G: |x|\le N\}$ satisfies
 $$u_m( y)\le C u_{n}(z) $$
 for all $m\in [N^2/8,N^2/4]$, $n\in [N^2/2,N^2]$
 and $y,z\in \{x\in G: |x|\le N/2\}$.
 \item [H\"older Estimate] Any bounded solution $u$ of {\em (\ref{diffeq})} in the discrete time cylinder $Q=[0,N^2]\times \{x\in G: |x|\le N\}$ satisfies
\begin{equation}\label{e:1.2}
| u_n(z)-u_m(y)|\le C\left[\left(|m-n|^{1/2}+|y^{-1}z|\right)/N\right]^\alpha
 \sup_{Q}\{|u|\}
 \end{equation}
 for all $m,n\in [N^2/8,N^2/2]$
 and $y,z\in \{x\in G: |x|\le N/2\}$.
 \end{description}
 \end{theo}

 In these two statements, the constants $C$ and $\alpha$ are independent of $N$ and of the solution $u$ (which can thus be translated both in time and in space if one so desires).
 In the context of parabolic differential equations, these estimates are the highlight of the celebrated
 De Giorgi-Nash-Moser theory.
 Informally, the first property (Parabolic Harnack Inequality) is the strongest as it (relatively easily) implies the second property (H\"older Estimate). The parabolic Harnack inequality also easily implies  the near diagonal two-sided estimate of Theorem \ref{th2}.

The goal of this work is to develop results such as
Theorems \ref{th2}  and \ref{T:1.4}
for walks on countable groups of polynomial volume growth when the measures driving the walks allow for a wide variety of long range jumps
and have infinite second moments.
For such random walks, it is known that a statement analogous to the above parabolic Harnack inequality cannot
 hold true. See, e.g.,  \cite{BC}. (Some integral version of the Harnack inequality,
 called a weak Harnack inequality, may hold in such cases; see \cite{ChaK}.)
 However, we will be able to prove a version of the near diagonal two-sided estimate of Theorem \ref{th2} and a H\"older
estimate \eqref{e:1.2} for globally
 bounded solutions of (\ref{diffeq}).  In both cases, the word-length geometry must be replaced by a geometry adapted to the
 long jump probability measure driving the random walk. See Theorems \ref{th-polmu}-\ref{thm0} and \ref{th-Holder}.

\subsection{Random walks with long range jumps}

In a finitely generated group of polynomial volume growth of degree $d$, all subgroups are finitely generated and have polynomial volume growth
of degree at most $d$.  This work focuses on a natural family of symmetric probability measures defined as follows. For a book treatment of the notion of regular variation, see \cite{BGT}.

\begin{defin} \label{def-Preg}
Let $G$ be a finitely generated group of polynomial volume growth.  Say a probability measure $\mu$ is in  $\mathcal P (G,\mbox{reg})$, if  there is an integer $k\ge 0$ such that $\mu$  can be written in the form
$$\mu=\sum_{i=0}^kp_i \mu_i,\;\; \sum_{i=0}^k p_i=1,\;\; p_0\ge 0,\;\;p_i>0,\;i=1,\dots, k,$$
where each $\mu_i$, $0\le i\le k$, is a symmetric probability measure on $G$ such that:
\begin{itemize}
\item The probability measure $\mu_0$ is finitely supported.
\item For each $1\le i\le k$, there exists a subgroup $H_i $ of $G$, equipped with a word-length $|\cdot|_i$ and  of polynomial volume growth of degree $d_i$,  and a function,
$\phi_i: [0,\infty)\ra (0, \infty) $, positive, increasing and of regular variation of positive index at infinity such that
\begin{equation} \label{def-mui}
\mu_i(h) \asymp  \left\{\begin{array}{cl} \left[\phi_i(1+|h|_i)  (1+|h|_i)^{d_i} \right]^{-1}  &\mbox{ if }  h\in H_i, \\
0&\mbox{ otherwise.}\end{array}\right. \end{equation}
\item There is an ${\eps} >0$ such that the finite set $\{g:\mu(g)>{\eps} \}$ generates $G$ and contains the identity element $e$.
\end{itemize}
\end{defin}
\begin{rem} When considering a measure $\mu$ in $\mathcal P(G,\mbox{reg})$, we will always assume that $\mu$ is given in the form
$\mu=\sum_{i=0}^kp_i \mu_i$
where the measures $\mu_i$, $1\le i\le k$, are described as in (\ref{def-mui}).  Hence, for any such $\mu$, we are given the subgroups $H_i$ and increasing
regularly varying functions $\phi_i$, $1\le i\le k$,  that are implicit in  the fact that $\mu$ is in $\mathcal P(G,\mbox{reg})$. By convention,
we set $H_0=G$ so that we have a well defined subgroup $H_i$ for each $i\in \{0,\dots,k\}$.
\end{rem}
\begin{rem} A measure $\mu$ in $\mathcal P(G,\mbox{reg})$ can be finitely supported if $k=0$ or if $k\ge 1$ and each subgroup $H_i$ is a finite subgroup of $G$
(and so $d_i=0$).
 When $k\ge 1$, the condition that $\mu(e)>0$ is automatically satisfied.
\end{rem}

The set   $\mathcal P(G,\mbox{reg})$  includes all  (non-degenerated) convex combinations of finitely many probability measures of the  power-law type
$$\mu_{H,\alpha}(h) \asymp \left\{\begin{array}{cl}(1+|h|_{H,S_H})^{-(d_H+\alpha_H)} & \mbox{ if } h\in H,  \;\;\;\;\alpha_H>0,\\
0& \mbox{ otherwise.}\end{array}\right. $$
Here, $H$ is a subgroup of $G$ with intrinsic volume growth of degree $d_H$. The subgroup $H$ and the positive real $\alpha_H$ can both vary freely and independently.
Note that our notion of  ``power-law type" is defined in reference to an intrinsic word-length $|\cdot|_{H,S_H}$ for the subgroup $H$ (here, $S_H$ is a fixed but arbitrary symmetric finite generating set for $H$).

More generally, simple examples of increasing functions of regular variation are
$$\phi(t)= (1+t)^\alpha[1+\log (1+ t)]^{\beta_1} [1+\log(1+\log (1+t))]^{\beta_2},$$ where $\alpha>0$ is the index of regular variation and $\beta_1,\beta_2\in \mathbb R$.
We refer the reader to \cite{BGT} for a detailed treatment of the notion of regular variation.  Some readers may prefer to restrict their attention to the simplest case $\phi(t)=(1+t)^\alpha$ as in the following theorem which illustrates one of the main
results of this paper.

\begin{theo} Let $G$ be a finitely generated group of polynomial volume growth. Let $\mu $ be a symmetric probability measure on $G$
which belongs to $\mathcal P(G,\mbox{\em reg})$  with $\phi_i(t)=t^{\alpha_i}$, $\alpha_i\in (0,2)$, $1\le i\le k$.  Then there exists a real
$d=d(G,\mu)\ge 0$ such that
$$
\forall\, n\in \{0,1,2, \dots,\},\;\;\; \mu^{(n)}(e)\asymp
\frac{1}{(1+n)^d}.
$$
\end{theo}

In fact we will prove a stronger
version of this theorem which deals with all measures in $\mathcal P_{\preceq} (G,\mbox{reg})$. This is a subset of $\mathcal P(G,\mbox{reg})$ whose definition involves a minor technical additional assumption regarding the functions $\phi_i$, $i\in \{1,\dots,k\}$ (see Definition \ref{def-Pregorder}).
In this more general
version, $$ \mu^{(n)}(e)\asymp \frac{1}{\mathbf F(n)},$$ where $\mathbf F$ is a regularly varying function which has positive index when $G$ is infinite. 
Our results allow for the explicit computation of
the index $d$
(more generally,  $\mathbf F$) in terms of the data describing the measure $\mu$ and the structure of the group $G$.  This is done by introducing  quasi-norms on $G$ that generalize the word-length (see Definitions \ref{def-quasi}-\ref{def-quasinorm}). Different measures typically call for different quasi-norms and for each measure $\mu$ in $\mathcal P_{\preceq}(G,\mbox{reg})$, we construct an adapted quasi-norm $\|\cdot\|$. Using this adapted quasi-norm, we prove a near diagonal two-sided estimate for $\mu^{(n)}$ and show that the bounded solutions of the associated parabolic difference equation are H\"older continuous.

The results proved here extend in significant ways  those obtained in \cite{SCZ-nil} by two of the authors. First, \cite{SCZ-nil} only deals with nilpotent groups. It is one of the main goals of this paper to treat the larger and more natural class of group of polynomial volume growth. Second, the measures considered in \cite{SCZ-nil} are convex combination of measures supported on {\em one parameter} discrete subgroups, i.e., subgroups of the type $\{g=s^m: m\in \mathbb Z\}$, $s\in G$. Here, we consider measures supported on general subgroups. Even when $G$ is nilpotent and the subgroups $H_i$ appearing in the definition of $\mu$ are one parameter subgroups, the present paper treats cases that where left aside in \cite{SCZ-nil}
(e.g., power laws
with arbitrary positive exponents). Nevertheless, some
of the main technical results of \cite{SCZ-nil} are used here again in a crucial way to pass from nilpotent groups to groups of polynomial volume growth.

\subsection{Dirichlet forms and spectral profiles} \label{sec-LambdaPhi}
We will make use of well established techniques based on Dirichlet forms and the notion of spectral profile. Let $\mu$ be a symmetric probability measure on a finitely generated group $G$.
We do not necessarily assume that the support of $\mu$ generates $G$.
The symmetric probability measure $\mu$ determines a Dirichlet form given by
$$\mathcal E_{G,\mu}(f,f)=\frac{1}{2}\sum_{x,y\in G}|f(xy)-f(x)|^2 \mu(y),\;\;f\in L^2(G).$$
Here, $L^2(G)$ is the Hilbert space  with norm $$\|f\|_2=\left(\sum _{x\in G}|f(x)|^2\right)^{1/2}.$$
The spectral profile of the measure $\mu$, $\Lambda_{2,G,\mu}$,  is the function defined over $[1,\infty)$  by
$$
\Lambda_{2,G,\mu}(v)= \min\left\{ \mathcal E_{G,\mu}(f,f)/ \|f\|_2^2 :  \, 1\le  \#\mbox{support}(f)\le v \right\}.
$$
It can also be defined
by considering each non-empty finite set $A\subset G$ of volume at most $v$, minimizing the Raleigh quotient of functions $f$ supported in $A$ to obtain the lowest eigenvalue $\lambda_\mu(A)$ of (minus) the discrete Laplacian, $f\mapsto f*(\delta_e-\mu)$, with Dirichlet boundary condition outside $A$, and taking the minimum of $\lambda_\mu(A)$ over all such finite sets $A$.
(Note that the discrete Laplacian $f\mapsto f*(\mu-\delta)$  is non-positive definite.)

In the cases of interest here,  we expect inverse power-function
estimates for $\Lambda_{2,G,\mu}$.  The well established
relation between the spectral profile
of $\mu$ and the decay of $\mu^{(2n)}(e)$  indicates that, for any $\gamma >0$,
\begin{itemize}
\item $\forall\,v\ge 1,\;\;\Lambda_{2,G,\mu}(v) \succeq v^{-1/\gamma}, $ is equivalent to  $\mu^{(2n)}(e) \preceq  n^{-\gamma}$.
\item $\forall\, v\ge 1,\;\;\Lambda_{2,G,\mu}(v) \preceq  v^{-1/\gamma}, $ is equivalent to  $\mu^{(2n)}(e) \succeq  n^{-\gamma}$.
\end{itemize}
More generally, if $ F$ is a positive monotone function of regular variation of index $\gamma>0$ (at infinity) and  $F^{-1}$ is
its inverse (hence a function of regular variation of index $1/\gamma$), then
\begin{itemize}
\item $\Lambda_{2,G,\mu}  \succeq  1/F^{-1}$ is equivalent to  $\mu^{(2n)}(e) \preceq  1/F(n)$.
\item $\Lambda_{2,G,\mu} \preceq   1/F ^{-1}$ is equivalent to  $\mu^{(2n)}(e) \succeq  1/F(n)$.
\end{itemize}
For details, see  \cite{CNash} and \cite[Section 2.1]{SCZ-aop}.

Another key property that we will use without further comment throughout is the fact that for any two symmetric probability measures $\mu_1,\mu_2$,
the inequality $$\mathcal E_{G,\mu_1} \le A \mathcal E_{G,\mu_2}$$
implies that
$$ \mu_2^{(2n)}(e)\preceq  \mu_1^{(2n)}(e);$$
that is, there
exist $A_1,A_2$ such that
$$\forall\, n=\{1,2,\dots\},\;\;\mu_2^{(2A_1 n)}(e)\le A_2 \mu_1^{(2n)}(e).
$$
In particular, if $\mu_1\asymp \mu_2$ on $G$ then $\mu_1^{(2n)}(e)\simeq \mu_2^{(2n)}(e)$. Whenever, in addition,  $\mu_1(e)\mu_2(e)>0$, the conclusion  easily extends to $\mu_1^{(n)}(e)\simeq \mu_2^{(n)}(e)$.
For background information on these notions and techniques, we refer the reader to the books \cite{VSCC,Woess} and to \cite{CNash,Hebisch1993,PSCstab,SCZ-aop}.

\subsection{Guide to the reader}
The paper is organized as follows.
Subsection \ref{S:2.1}
introduces the quasi-norms and geometries  that are key to the study of the walks driven by measures in $\mathcal P_ {\preceq} (G,\mbox{reg})$.  See Definition
\ref{def-Pregorder}.
Each of these geometries is associated with a generating tuples $\Sigma=(s_1,\dots,s_k)$ of elements of the group $G$ and a weight function system $\mathfrak F=\{F_s, s\in \Sigma\}$. In the study of random walks, the structure of a given measure $\mu$ in $\mathcal P _{\preceq} (G,\mbox{reg})$ will determine in large part how to choose
$\Sigma$ and $\mathfrak F$.

Subsection \ref{sec-scznil}
 describes results from \cite{SCZ-nil} concerning the case of nilpotent groups which play a key role in the rest of the paper.  See Theorem \ref{th-nilgeom}.

Section \ref{sec-up}
discusses  how geometric results (existence of coordinate-like
systems and volume growth) leads to lower bounds on the spectral profile and upper bounds on the probability of return of measures in $\mathcal P _ {\preceq} (G,\mbox{reg})$.  Sub-section \ref{sec-nil} applies these results to nilpotent groups. Sub-section \ref{sec-polvol}, one of the most important parts
 of the paper, explains how to obtain sharp results in the case of groups of polynomial volume growth. Given a group of polynomial growth and a measure $\mu\in \mathcal P_ {\preceq}  (G,\mbox{reg})$, we explain the construction of a well adapted geometry on $G$ based on the (well-known) existence of a nilpotent group $N$ with finite index in $G$. In fact, we construct geometries on $N$ and on $G$ which are closely related   to each other and well adapted to the given measure $\mu $ on $G$.  Some explicit examples are given.

Section \ref{sec-low}  provides  matching upper-bounds on the spectral profiles  and the corresponding lower bounds on the probability of return. This is done by providing appropriate test functions which are defined using the quasi-norms of section \ref{sec-G}. See Theorem \ref{th-polmu}.

Section \ref{sec-control} contains one of the main theorems, Theorem \ref{thm0}, which gather the main properties of the iterated convolution $\mu^{(n)}$
and the associated random walk when $\mu\in \mathcal P_{\preceq}(G, \mbox{reg})$ and $G$ has polynomial volume growth.

Section \ref{sec-Holder}  proves the
H\"{o}lder  estimate for  solutions of the corresponding  discrete parabolic equation (see, Theorem \ref{th-Holder}).
The main results of the paper are in Theorems
 \ref{th-polmu}, \ref{thm0} and \ref{th-Holder}.

\section{Geometries for
random walks with long range jumps} \label{sec-G} \setcounter{equation}{0}

As noted in the introduction, the word-length associated to a finite symmetric generating set $S$ is a key element in developing an understanding of the behavior of the random walks driven by symmetric finitely supported measures. The question arises as to what are the natural geometries that might help us understand random walks that allow for long range jumps.   This section introduces such geometries.

\subsection{Weight systems and quasi-norms}\label{S:2.1}

First, let us give a more formal definition of the word-length associated with a finite set of generator.   Fix a finite alphabet $\Sigma=\{s_1,\dots, s_k\}$
and adjoin to it the formal inverses (new letters) $\Sigma^{-1}=\{s_1^{-1},\dots,s_k^{-1}\}$.    A finite word $\word$ over $\Sigma\cup \Sigma^{-1}$ is
a formal product (i.e., a finite sequence) $\word=\sigma_1\dots \sigma_m$ with $\sigma_i\in \Sigma\cup \Sigma^{-1}$, $1\le i\le m$. Equivalently, we can write
$\word=\sigma_1^{{\eps} _1}\dots \sigma_m^{{\eps} _m}$ with $\sigma_i\in \Sigma$ and ${\eps} _i\in \{\pm 1\}$, $1\le i\le m$.
If $G$ is a group which contains  elements called $s_1,\dots, s_k$, we say that the word $\word=\sigma_1\dots\sigma_m$  over $\Sigma\cup \Sigma^{-1}$
is equal to $g\in G$, if
$\sigma_1\dots\sigma_m=g$  when reading this product in $G$.
Formally, one should denote the letters by $\mathbf s_i$, the corresponding group elements by $s_i$, and introduce the map $\pi:\cup_{q=0}^\infty(\Sigma\cup\Sigma^{-1})^q\ra G$  defined by  $\pi(\boldsymbol{\sigma}_1\dots\boldsymbol{\sigma}_m)=\sigma_1\dots\sigma_m$.
With this notation the word-length $|g|$ of an
element $g\in G$ with respect to the $k$-tuple of generators generators $(s_1,\dots,s_k)$ and their inverses is
$$|g|=\inf\{ m : \exists \,\word\in (\Sigma\cup \Sigma^{-1})^m, \;g=\word \mbox{ in } G \} .$$
By convention, $|e|=0$ ($e$ can be obtained as the empty word).
For illustrative purpose, we  introduce the following variant
$$\|g\|=\inf\left\{ \max_{s\in \Sigma}\{\mbox{deg}_s(\word)\}:\word\in \cup_0^\infty (\Sigma\cup \Sigma^{-1})^m,\;\; g=\word \mbox{ in } G\right\},$$
where, for each $s\in \Sigma$ and $\word \in \cup_0^\infty (\Sigma\cup \Sigma^{-1})^m$,  $\word= s_{j_1}^{{\eps} _1}\dots s_{j_m}^{{\eps} _m}$,
we set
$$
\mbox{deg}_s(\word)= \#\{\ell\in \{1,\dots ,m\}:  s_{j_\ell}=s\}.
$$
In words, $\mbox{deg}_s(\word)$ is the number of times the letter $s$ is used (in the form $s$ or $s^{-1}$) in the word $\word$.
Obviously,
$$ \|g\|\le |g| \le k \|g\|, \; \mbox{ where } \;k=\#\Sigma.$$
The reader should note that when defining $\mbox{deg}_s$,
we think of $s$ as a letter in the alphabet $\Sigma$ (two distinct words consisting of letters in $\Sigma\cup \Sigma^{-1}$ might become equal
as an element in $G$). In addition, $\mbox{deg}_s$ counts the occurrences of both $s$ and $s^{-1}$. For instance, consider the word $\word=s_1s_2s_1^{-1}s_2^{-1}s_3s_1^{-1}$.  The degrees are as follows:
$$\mbox{deg}_{s_1}(\word)=  3, \;\;\mbox{deg}_{s_2}(\word)=2,\;\; \mbox{deg}_{s_3}(\word)=1.$$
This is the case even if it happens, as it may, that $s_3=s^{-1}_1$ in $G$.

\begin{defin}\label{def-quasi}
We say that a map $N: G\ra [0,\infty)$ is a norm, if
$$N(gh)\le N(g)+ N(h).$$
We say that it is a quasi-norm, if there exist a constant $A$ such that
$$N(gh)\le A(N(g)+N(h)).$$  \end{defin}
\begin{rem} The quasi-norms constructed in this paper have two additional properties. They are symmetric ($N(g)=N(g^{-1})$, $g\in G$) and $N(e)=0$.
\end{rem}

\begin{exa}
The maps  $g\mapsto |g|$ and $g\mapsto \|g\|$ associated to a generating tuple $(s_1,\dots,s_k)$ as above are norms.
\end{exa}

Now, we introduce a  (potential) quasi-norm  $\| \cdot \|_\mathfrak F$ associated with a family $\mathfrak F$ of continuous
and strictly increasing functions on $[0, \infty)$.  This will be a quasi-norm under some  additional technical assumptions on the family $\mathfrak F$.
 The  basic data for such a function  $\| \cdot \|_\mathfrak F$  consists of
a group $G$, a tuple $\Sigma=(s_1,\dots,s_k)$ (abusing notation, we will consider each $s_i$ both as an abstract symbol (letter) and as a group element in $G$) and a family $\mathfrak F$ of continuous
and
strictly
 increasing functions
$$F_s: [0,\infty)\ra [0,\infty) ,\;\;s\in \Sigma,\;\; F_s(t)\asymp t \mbox{ on } [0,1],$$
with the property that for $s,s'\in \Sigma$,
\begin{equation} \label{order}
\mbox{ either }F_s\preceq F_{s'} \mbox{ or    } F_{s'}\preceq F_s  \mbox{ on a neighborhood of infinity}.
\end{equation}
With proper care and technical modifications,
condition \eqref{order}
can probably be removed but we will assume it holds throughout this paper.  Each of the function $F_s$ is invertible, and we denote by $F_s^{-1}$
its inverse.   A good example to keep in mind is the case when, for each $s\in \Sigma$, we are given a positive real $\weight (s)$ and $F_s(t)=t\mathbf 1_{[0,1]}(t)+
t^{\weight (s)}\mathbf 1_{(1,\infty)}(t)$   (or, more or less equivalently,  $F_s(t)=(1+t)^{\weight (s)}-1$). We think of
$F_s$ as a weight function assigned to $s \in \Sigma$.

\begin{defin} \label{def-quasinorm}
Given $G$, $\Sigma$ and $\mathfrak F$ as above, for each element $g\in G$, set
$$\|g\|_\mathfrak F=
\inf\left\{ \max_{s\in \Sigma}\{ F_s^{-1}(\mbox{deg}_s(\word))\}:\word\in \cup_0^\infty (\Sigma\cup \Sigma^{-1})^m,\;\; g=\word \mbox{ in } G\right\}.$$
By convention, $\|e\|_{\mathfrak F}=0$. If $g$ cannot be represented as a finite word over $\Sigma\cup \Sigma^{-1}$, set $\|g\|_{\mathfrak F}=\infty$.
\end{defin}
In other word,  $\|g\|_{\mathfrak F}$ is the least $R$ such that there exists a finite word $\word$ such that  $\word=g$ in $G$ and
$$\mbox{deg}_s(\word)\le F_s(R) \mbox{ for each } s\in \Sigma.$$
This last inequality indicates that
each  letter $s$ (in the form $s$ or $s^{-1}$) in $\Sigma$
 is used at most $F_s(R)$ times in the word $\word$.

\begin{rem} \label{rem-convex} In the context of nilpotent groups,
this definition of $\| \cdot \|_\mathfrak F$
appears in   \cite[Definition 2.8]{SCZ-nil}.\end{rem} 
\begin{rem} If each $F_s$ satisfies $F^{-1}_s(t_1+t_2)\le A(F^{-1}_s(t_1)+F^{-1}_s(t_2))$, then $\|\cdot\|_{\mathfrak F} $ is a quasi-norm (a norm, if $A=1$).
In particular, $\|\cdot \| _{\mathfrak F}$ is a norm, if each $F_s$ is convex.
\end{rem}
\begin{rem} \label{rem-circF}
If each $F_s$ is replaced by  $\widetilde{F}_s=F_s \circ F^{-1}$  for some continuous
and
strictly
increasing function $F:[0,\infty)\ra [0,\infty)$ with
$F(t)\asymp t$ on $[0,1]$, then   $$\|g\|_{\widetilde{\mathfrak F}}=F(\|g\|_{\mathfrak F}) \;\mbox{ for each } g\in G .$$
\end{rem}

\begin{rem}  \label{rem-doubling} Say that  a non-negative function $f$ defined on $(0,\infty)$ is doubling, if there exists a constant $A_f>0$ such that
\begin{equation} \label{f-doubling}
\forall \,t> 0,\;\; f(2t)\le A_f f(t)   \mbox{ and } 2f(t)\le f(A_ft).
\end{equation} Over the class of  doubling functions, the
equivalence relations $\simeq $ and $\asymp$ coincide. Suppose that
we have two weight functions systems $\mathfrak F$ and $\mathfrak
F'$ define over $\Sigma$, and that all functions $F_s$ and $F'_s$
are doubling. Suppose further that for each $s\in \Sigma$,
$F_s\simeq F'_s $. Then we can conclude that $\|g\|_{\mathfrak F}
\asymp \|g\|_{\mathfrak F'}$ over $G$.
\end{rem}
\begin{exa} Let $G=\mathbb Z^k$ with the canonical generators $(s_1,\dots,s_k)$. For $t\ge 1$, let $F_{s_i}(t)=t^{\weight_i}$ with $\weight_i>0$. Then, for $x=(x_1,\dots,x_k)=\sum x_is_i$,
$$\|(x_1,\dots, x_k)\|_{\mathfrak F}=\max_i\{ |x_i|^{1/\weight_i}\}.$$
\end{exa}
\begin{exa}[Heisenberg group]  \label{exa-Heisenberg}
Let $G=(\mathbb Z^3,\bullet)$ with
$$g\bullet g'= (x_1+x_1',x_2+x_2', x_3+x_3'+ x_1x_2'),$$   i.e., in coordinates, matrix multiplication in the Heisenberg group
$$G=\mathbb H(3,\mathbb Z)=\left\{g=(x_1,x_2,x_3)=\left(\begin{array}{ccc} 1& x_1& x_3\\0&1& x_2\\0& 0& 1\end{array}\right): x_1,x_2,x_3\in \mathbb Z\right\}.$$
For $i=1, 2 ,3$, let $s_i$ be the triplet with a $1$ in position $i$
and $0$ otherwise. For $t\ge 1$, let $F_{s_i}(t)=t^{\weight_i}$ with
$\weight_i>0$. Then
$$\|(x_1,x_2,x_3)\|_{\mathfrak F}\asymp \left\{\begin{array}{cl} \max\{ |x_1|^{1/\weight_1},|x_2|^{1/\weight_2},|x_3|^{1/\weight_3}\}   &\mbox{ if } \weight_3\ge \weight_1+\weight_2,\\
\max\{|x_1|^{1/\weight_1}, |x_2|^{1/\weight_2},|x_3|^{1/(\weight_1+\weight_2)}\} & \mbox{ if }  \weight_3\le \weight_1+\weight_2.
\end{array}\right.$$
See \cite[Examples  1.1 and  4.3]{SCZ-nil}.
\end{exa}

The following proposition is technical in nature. Parts (b) and (c) will be used later in deriving the main new results
of this paper.

\begin{pro} \label{pro-variation}
Consider a weight function system $(\mathfrak F,\Sigma)$ on a group $G$ as above. Assume that for
each $s\in \Sigma$,  the function $F_s$ is regularly varying of index $\weight (s)>0$ at infinity.
\begin{itemize}
\item[$(a)$] There exists a weight function system $(\mathfrak F_0,\Sigma)$  such that
each member  $F_{0,s}$
is in $\mathcal C^1([0,\infty))$,
increasing, and of smooth variation in the sense of {\em
\cite[Section 1.8]{BGT}} with $F_{0,s}\asymp F_s$,
for $s\in \Sigma$.

\item[$(b)$] 
For any fixed $\weight^*\in (0, \infty)$ with $\weight^*>\max\{\weight (s): s\in \Sigma\} $,
there is a weight function system $(\mathfrak F_1,\Sigma)$ such that
each member $F_{1,s}$
is in $\mathcal C^1([0,\infty))$,
increasing, and of smooth variation of index less than $1$ in the
sense of {\em \cite[Section 1.8]{BGT} } with
$F_{1,s} \asymp F_s\circ F^{-1}$, where
$F(t)= (1+t)^{\weight^*}-1$. In particular, there
exists a positive real $A$ such that, for all $s\in \Sigma$  and all
$T\in [0,\infty)$,
\begin{equation}\label{Fder}
\sup_{[0,T]}\left\{ \frac{d F_{1,s}(t)}{dt}\right\} \le A \frac{F_{1,s}(T)}{T}
 \end{equation}
and
$$
\|g\|_{\mathfrak F}\asymp \|g\|_{\mathfrak F_1}^{1/\weight^*}
\quad \mbox{over } G.
$$

\item[$(c)$]
For any fixed $\weight_*$ with $0<\weight_*< \min\{\weight (s): s\in \Sigma\}$, there
is a weight function system $(\mathfrak F_2,\Sigma)$ such that
each member
$F_{2,s}$  is  in $\mathcal C^1([0,\infty))$
increasing, convex, and of smooth variation in the sense of {\em
\cite[Section 1.8]{BGT} } with
$F_{2,s} \asymp F_s\circ F^{-1}$, where
$F(t)= (1+t)^{\weight_*}-1$. In particular, $g\mapsto \|g\|_{\mathfrak
F_2}$ is a norm and
$$\|g\|_{\mathfrak F}\asymp \|g\|_{\mathfrak F_2}^{1/\weight_*} \mbox{ over } G.$$
\end{itemize}
\end{pro}

\begin{proof}  Part (a) is essentially  \cite[Theorem 1.8.2]{BGT}. The difference is that we impose some simple additional conditions regarding the behavior of
$F_{0,s}$ on  $[0,a]$ for some $a>0$ (smooth regular variation is a
property of $F_{0,s}$ on a neighborhood of infinity). By inspection,
it is clear that these additional  conditions can be achieved.

The main point of part (b) is that the functions  $ F_s\circ
F^{-1}$, $s\in \Sigma$, are all regularly varying with
positive index strictly less than one. By \cite[Theorem 1.8.2]{BGT},
there are positive functions $\widetilde{F}_s$  (defined on a
neighborhood $[a,\infty)$ of infinity), increasing, of smooth
regular variation and satisfying (see the discussion on \cite[Page
44]{BGT}) $\widetilde{F}_s \sim F_s\circ F^{-1}$  at infinity  and
$$\frac{d\widetilde{F}_s(t)}{dt} \le A\frac{\widetilde{F}_s(t)}{t}
\mbox{ on } [a,\infty).$$ We can now pick a constant $C(s)$ so that
the function $F_{1,s}$ obtained by extending  $C(s)+\widetilde{F}_s$
linearly on $[0,a]$, so that $F_{1,s}(0)=0$,
$F_{1,s}(t)=C(s)+\widetilde{F}_s(t)$ on $[a,\infty)$, which belongs
to $\mathcal C^1[0,\infty)$, is increasing  and  of smooth regular
variation,  and satisfies the other desired properties.

The main point of part (c) is that the functions  $ F_s\circ
F^{-1}$,  $s\in \Sigma$, are now all regularly varying with
positive index strictly greater than one. In this case,
\cite[Theorem 1.8.2]{BGT} gives positive  functions
$\widetilde{F}_s$ (defined on a neighborhood $[a,\infty)$ of
infinity), increasing, convex, and of smooth regular variation  such
that $\widetilde{F}_s \sim F_s\circ F^{-1}$. Proceeding as in part
(b), we can extend modified versions
to $[0,\infty)$
with all the desired properties.
\end{proof}

\begin{rem}
 In parts (b) and (c) of Proposition \ref{pro-variation},
we are avoiding the slightly troublesome case when the index is  exactly $1$.
This is troublesome when the corresponding weight function is not
exactly linear. For instance, in part (b), the result is still
correct, but the derivative and its upper bound are not necessarily
monotone. This becomes a real problem for part (c).    By a further
variation of this argument, one can use composition by a function
$F$  as above to avoid all integers index (indeed, there is only
finitely many $F_s$ to deal with so proper choices of $\weight_*$ and
$\weight^*$ do the trick).  
Then one can apply \cite[Theorem 1.8.3]{BGT},
which is more elegant than the above construction and provides
similar results.
\end{rem}

\subsection{Volume counting from \cite{SCZ-nil}} \label{sec-scznil}
Given a group $G$ equipped with  a generating tuple $(s_1,\dots,s_k)$ and a weight function system $\mathfrak F$, it is really not clear how to  compute or ``understand" the map $g\mapsto \|g\|_{\mathfrak F}$. The article \cite{SCZ-nil}  considers  the case of nilpotent groups and connects the results
to the study of certain
random walks with long range jumps.

Beyond the nilpotent case, questions such as
\begin{itemize}
\item What is the  cardinality of $\{g\in G: \|g\|_{\mathfrak F}\le R\}$?
\item Which choice of $(\Sigma, \mathfrak F)$ is relevant to which
random walk with long range jumps?
\end{itemize}
do not seem very easy to answer.

For a better understanding of our main results, it is useful to
review and emphasize the main volume counting results derived in
\cite{SCZ-nil} in the case of nilpotent groups. We want to apply
these results in the context of Definition \ref{def-quasinorm}. 
We make the assumption that
 all the functions appearing in the system
$\mathfrak F$ are doubling (see Remark \ref{rem-doubling}). Recall
that, by (\ref{order}), we have a well defined total order on
$\{F_s, s\in\Sigma\}$ (modulo the equivalence relation $\asymp$ on
a neighborhood of infinity. The equivalence relation $\simeq $ and $\asymp$ are equal on doubling functions).

 Following \cite{SCZ-nil}, we extend our given weight function system to the collection of all finite
 length abstract commutators over the alphabet  $\Sigma\cup\Sigma^{-1}$
 by using the rules
$$
F_{s^{-1}}=F_{s}  \mbox{ for  } s\in \Sigma \quad \hbox{and} \quad
 F_{[c_1,c_2]}= F_{c_1}F_{c_2}.
$$
In short, abstract commutators are formal entities obtained by induction via the building rule $[c_1, c_2]$ starting from $\Sigma\cup \Sigma^{-1}$ (see \cite{SCZ-nil} for more details). Observe that  the family of functions $F_c$,
$c$ running over formal commutators,
 have  property (\ref{order}) and thus carry a well defined total order (again, modulo the equivalence relation $\asymp $).  For notational convenience, we introduce formal representatives
for the linearly ordered distinct elements of $\{ F_c \mbox{ mod } \asymp \}$ and call these representatives
$$\bar{\weight}_1<\bar{\weight}_2< \bar{\weight}_3 <\cdots .$$   Hence,  $\bar{\weight}_1$ represents the $\asymp $ equivalence class associated with the smallest of the weight function $F_c$, etc.   For each $\bar{\weight}_i$, let $\mathbf F _i$ be a representative of  the $\asymp$ equivalence class of functions
$F_c$ associated with $\bar{\weight}_i$.  By definition, we can pick any  commutator $c$ with $F_c\in \bar{\weight}_i$ and set
$$\mathbf F_i= \prod_1^\ell F_{\sigma_i}, $$
where $\sigma_1,\dots, \sigma_\ell$ is the complete list (with repetition) of the elements of $\Sigma\cup \Sigma^{-1}$ that are used to form the formal commutator $c$.

\begin{defin}Referring to the above notation,  let   $G^\mathfrak F_i$ be the subgroup of $G$ generated by all (images in $G$ of) formal commutators   $c$  such that
$$ F_c   \in  \cup_{j\ge i} \bar{\weight}_j .$$
In other words, $G^\mathfrak F_i$ is generated by all commutators such that $F_c \succeq \mathbf F_i$.
\end{defin}
Obviously, these groups form a descending sequence of subgroups of $G$ and, under the assumption that $G$ is nilpotent, there exists a smallest integer $j_*=j_*(\mathfrak F)$ such that $G^\mathfrak F_{j_*+1}=\{e\}$.  Further,  not only $G^\mathfrak F_j\supseteq G^{\mathfrak F}_{j+1}$ but also
(see \cite[Proposition 2.3]{SCZ-nil})
$$[G,G^\mathfrak F_j]\subseteq G^\mathfrak F_{j+1},$$
so that $G^\mathfrak F_j/G^\mathfrak F_{j+1}$ is a finitely
generated abelian group.  We let
$$\mathfrak r_j=\mbox{rank}(  G^\mathfrak F_j/G^\mathfrak F_{j+1})$$  be the torsion free rank of this
 abelian group  (by the finitely generated abelian group structure theorem,
 any such group  $A$ is isomorphic to a product of the form   $K\times \mathbb Z^r$ where $K$ is a finite abelian group. The integer $r$ is the torsion free rank of the group $A$).
\begin{theo}[{\cite[Theorems 2.10 and 3.2]{SCZ-nil}}]  \label{th-nilgeom}
Let $G$ be a nilpotent group equipped with a finite tuple of generators $\Sigma$ and a weight function system
$\mathfrak F$ as above satisfying {\em (\ref{order})-(\ref{f-doubling})}. From this data, extract the functions $\mathbf F_j$  and integers $\mathfrak r_j$, $1\le j\le j_*$, as explained above.  Then
$$\#\{g\in G: \|g\|_{\mathfrak F} \le R\} \asymp  \prod_1^{j_*} [\mathbf F_i(R)]^{\mathfrak r_j}.$$
Furthermore, there exist an integer $Q$, a constant $C$,  and a
sequence $\sigma_1,\dots, \sigma_Q$ with $\sigma_j\in \Sigma$,
$1\le j\le Q$, such for any positive $R$ and  any element $g$ with $\|g\|_{\mathfrak F}
\le R$, $g$ can be written in the form
$$g=\prod _{j=1}^Q \sigma_j^{x_j} \mbox{ with }  |x_j|\le  CF_{\sigma_j}(R).$$
\end{theo}
\begin{defin}\label{def-volnilF}
 Let $G$ be a nilpotent group equipped with a finite tuple of generators $\Sigma$ and a weight function system
$\mathfrak F$ as above satisfying  (\ref{order})-(\ref{f-doubling}). From this data, extract the functions $\mathbf F_j$  and integers $\mathfrak r_j$ , $1\le j\le j_*$, as explained above.  Set $$\mathbf  F_{G,\mathfrak F}=\mathbf F_{\mathfrak F} = \prod_1^{j_*} \mathbf F_i^{\mathfrak r_j}.
$$
\end{defin}
By Theorem \ref{th-nilgeom},  for any nilpotent group and weight function system satisfying (\ref{order})-(\ref{f-doubling}),  we have the explicit volume estimate
$$ \#\{g\in G: \|g\|_{\mathfrak F}\le R\} \asymp \mathbf F_{\mathfrak F}(R).$$

\begin{exa}  Let us return to the Heisenberg group example $G=\mathbb H(3,\mathbb Z)$, Example \ref{exa-Heisenberg},  with $F_{s_i}(t)=t^{\weight_i}$ with
$\weight_i>0$ for all $t\ge 1$ and $1\le i\le 3$. Then
$$\mathbf F_{\mathfrak F}(R)\asymp \left\{\begin{array}{cl}  R^{\weight_1+\weight_2+\weight_3}  &\mbox{ if } \weight_3\ge \weight_1+\weight_2,\\
 R^{2 (\weight_1+\weight_2) }& \mbox{ if }  \weight_3\le \weight_1+\weight_2.
\end{array}\right.$$
\end{exa}

\section{Upper bounds on return probabilities}    \label{sec-up}
\setcounter{equation}{0}

\subsection{A general approach}
In this section we discuss how to obtain upper bounds for the return probability of a measure $\mu\in \mathcal P_ {\preceq} (G,\mbox{reg})$, a subset of $
\mathcal P(G,\mbox{reg})$ which is described below in Definition
 \ref{def-Pregorder}.  The main tool is the following technical result.  For the proof, we can follow the proofs of \cite[Theorems 4.1 and 4.3]{SCZ-nil} with minor adaptations.
\begin{pro} \label{pro-PS1}
Let $G$ be a countable group equipped with symmetric probability measures $\mu_i$, $0\le i\le k$.  Assume that there exists a constant $C>0$
such that for each $R>0$ and $0\le i\le k$, there is a subset $K_i (R)\subset G$ such that
\begin{equation}\label{PSR}
\sum_{x\in G}|f(xh)-f(x)|^2 \le CR \mathcal E_{G,\mu_i}(f,f),\;\; f\in L^2(G),\; h\in K_i(R).
\end{equation}
Assume further that there are an integer $Q$ and a positive monotone function $\mathbf F$ of regular variation of positive index with inverse $\mathbf F^{-1}$
 such that
$$
\# \left(\cup_{i=0}^k K_i(R)\right)^Q\ge  \mathbf F(R),
$$
where $ \left(\cup_{i=0}^k K_i(R)\right)^Q =\{g=g_1\dots g_Q: g_i\in \cup_{i=0}^k K_i(R) \}$ is viewed as a subset of  $G$. Set
$\mu=(k+1)^{-1}\sum\limits_{i=0}^k \mu_i$.
 Then
 $$
   \Lambda_{2,G,\mu}\succeq 1/\mathbf F^{-1}  \, \quad
  \hbox{and} \quad\mu^{(2n)}(e)\preceq 1/\mathbf F(n).$$
\end{pro}
To understand how this proposition works in practice, note that the larger the set $K_i(R)$ is, the harder it is to prove (\ref{PSR}), but the  faster  the growth from the lower-bound $F(R)$
on  $\# \left(\cup_{i=0}^k K_i(R)\right)^Q$ one might expect. Observe also that the inequality (\ref{PSR}) is trivial if $K_i(R)=\{e\}$ but, then, $K_i(R)$ does not contribute at all to the growth of $\# \left(\cup_{i=0}^k K_i(R)\right)^Q$. If $K_i(R)=\{e\}$ for all $i=0,\dots,k$, then $F(R)$ cannot grow and the conclusion of the proposition is trivial. Formally, this case is excluded by the requirement that $F$ is regularly varying of positive index.

The next two propositions provide a way to verify assumption (\ref{PSR}) for the type of measures of interest to us here.
\begin{pro} \label{pro-PS2}
Let $G$ be a countable group equipped with a symmetric probability measure $\mu$ supported on a subgroup $H$ equipped with a quasi-norm $\|\cdot\|$
such that there exist a constant $d>0$ and a positive monotone regularly varying function $\phi: (0,\infty)\ra (0,\infty)$ of positive index at infinity such that
$$\#\{h\in H:  \|h\|\le r\}\asymp r^d,\;\;\;\mu(h) \asymp  \left[\phi(1+\|h\|)  (1+\|h\|)^d \right]^{-1}.$$
Then there is a constant $C$ such that, for all $f\in L^2(G)$, $R\ge 1$ and $ h\in H$ with $\|h\|\le R$, we have
\[
\sum_{x\in G}|f(xh)-f(x)|^2 \le C\phi(R) \mathcal E_{G,\mu}(f,f).
\]
\end{pro}
\begin{proof}  First observe that
\begin{equation}\label{eq:feobaqz}
\sum_{\|h\|\ge r}\mu(h) \asymp 1/\phi(r).
\end{equation}
This is proved by summing over $A$-adic  annuli with $A$ large enough so that
$$\#\{A^{n-1}\le \|h\|\le A^{n}\}\asymp \#\{ \|h\|\le A^n\} \asymp A^{dn}.$$
Next,  note that for $h,h'\in H$, the inequality $\|h'\|\ge
C_0\|h\|$
 (with $C_0\in(0,1/A)$ small enough, where $A$ is the constant in the definition of quasi-norm, see Definition \ref{def-quasi})
  implies $$\|h^{-1}h'\|\ge
C^{-1}\|h'\|\ge C^{-1}C_0\|h\| \mbox{ and } \mu(h')\le C \mu(h^{-1}h')
$$ for some constant $C\ge 1$.
 Now, write
\begin{align*}
\lefteqn{ \sum_{x\in G}|f(xh)-f(x)|^2  \left(\sum_{|h'|\ge
C_0|h|} \mu(h')\right)  } \hspace{.5in}&  &\\
&\le 2\sum_{|h'|\ge
C_0|h|}\sum_{x\in G}( |f(xh)-f(xh')|^2 + |f(xh')-f(x)|^2)\mu(h')\\
&\le  C \sum_{h'\in H: |h'|\ge
C_0|h|}\sum_{x\in G}|f(x)-f(xh^{-1}h')|^2\mu(h^{-1}h')\\
&\quad+
2\sum_{h'\in H,x\in G}|f(xh')-f(x)|^2\mu(h') \\
&\le  C \sum_{h'\in H: |h^{-1}h'|\ge
C^{-1}|h'|}\sum_{x\in G}|f(x)-f(xh^{-1}h')|^2\mu(h^{-1}h')\\
&\quad+
2\sum_{h'\in H,x\in G}|f(xh')-f(x)|^2\mu(h') \\
&\le   (2+C) \mathcal E_{G,\mu}(f,f).
\end{align*}
This gives the desired inequality.
\end{proof}

\begin{pro} \label{pro-PS3}
Let $G$ be a countable group equipped with symmetric probability measure $\mu$ supported on a subgroup $H$ equipped with a finite generating set and its word length $|\cdot|$.  Assume that there  exist a constant $d>0$ and a positive monotone regularly varying function $\phi: (0,\infty)\ra (0,\infty)$  of positive index at infinity such that
$$\#\{h\in H:  |h|\le r\}\asymp r^d,\;\;\;\mu(h) \asymp  \left[\phi(1+|h|)  (1+|h|)^d \right]^{-1}.$$
Then there is a constant $C$ such that, for all $f\in L^2(G)$, $R\ge 1$ and $ h\in H$ with $\|h\|\le R$, we have
\[
\sum_{x\in G}|f(xh)-f(x)|^2 \le C\Phi(R) \mathcal E_{G,\mu}(f,f),
\]
where $\Phi(t)= t^{2}/\int_0^t \frac{sds}{\phi(s)}$, $t\ge 1$.
\end{pro}
\begin{proof}  Let $u_r $ be the uniform probability measure on $\{h\in H: |h|\le r\}$.
Follow the proof of \cite[Proposition A.4]{SCZ-aop} to show that for any $f\in L^2(G)$ and any $0<s<r<\infty$ and $h\in G$ with $|h|\le r$, we have
$$\sum_{x\in G}|f(xh)-f(x)|^2\le C (r/s)^2\sum_{x\in G}\sum _{h\in H}|f(xh)-f(x)|^2u_s(h).$$
Observe that $\mu\asymp \sum_0^\infty \frac{1}{\phi(2^n)} u_{2^n}$
and that, for $r\ge 1$,
$$\sum_{n: 2^n\le r} \frac{2^{2n}}{\phi (2^n) } \asymp \int_0^r \frac{sds}{\phi(s)}.$$
The desired inequality follows.
\end{proof}

\begin{rem}  If  $\phi$ is regularly varying of positive index $\gamma$, then we always have
that   $\Phi (t)\le C\phi(t)$ and $\phi(t) \asymp \Phi(t)$ if
$\gamma \in (0,2)$.  If $\gamma>2$, $\Phi(t)\asymp t^2$ which is
much less than $\phi$ on $(1,\infty)$. \end{rem}

\begin{rem}
The proof of Proposition
\ref{pro-PS3} outlined above uses  the fact that,  roughly
speaking,  an element $h$ with $|h|=r$ can be written as a product
of $(r/s)$ elements of length  at most $s$.  This property is not
necessarily true for an arbitrary quasi-norm $\|\cdot\|$ as in
Proposition \ref{pro-PS2}.
\end{rem}

\begin{defin} \label{def-PregPhi}
Given $\mu=\sum_0^k \mu_i \in  \mathcal P (G,\mbox{reg})$ with  $\mu_i$ defined in terms of a regularly varying function $\phi_i$
as in  (\ref{def-mui}), $1\le i\le k$,  set
$$\Phi_0(t)= \max\{t,t^2\}$$
and, for $1\le i\le k$,
 $$\Phi_i(t)= \frac{t^2}{\int_0^t \frac{2s}{\phi_i(s)} ds}\;\; \mbox{ on } \;[1,\infty),$$
 with $\Phi_i$ extended linearly on $[0,1] $ with $\Phi_i(0)=0$.
\end{defin}

\begin{rem} \label{rem-Phi}
The exact definition of $\Phi_i$ in the interval $[0,1]$ is not very important to us. For convenience, we prefer  to have it vanish linearly at $0$.
Because $\phi_i$ is increasing,  one can check that $$ \Phi_i\le \phi_i  \mbox{ on } [1,\infty) \; \mbox{ and }\Phi_i \mbox{ is increasing on } (0,\infty).$$
 Further, $\Phi_i$
is regularly varying at infinity of index  in $ (0, 2]$. More precisely,  $\Phi_i \asymp \phi_i$ when the index of $\phi_i$ is in $(0,2)$. When the index of $\phi$ is at least $2$, then  $\Phi(t)\asymp t^2/\ell(t)$ where $\ell$  is increasing and slowly varying at infinity (i.e., index $0$).  In particular, in all cases, $\Phi_i(t) \preceq \max\{t, t^2\}=\Phi_0(t)$, $1\le i\le k$.
\end{rem}

\begin{pro} \label{pro-PS4}
Let $G$ be a countable group. Let $\mu \in \mathcal P(G,\mbox{\em
reg})$ and, referring to
 {Definitions $\ref{def-Preg}$ and $\ref{def-PregPhi}$}, set
$$ K_i(r)=\{h\in H_i:  \Phi_i(|h|_i) \le r\},\;\; 0\le i\le k$$ and, for some fixed integer $Q$,
$$K(r)=\left\{g\in G: \;g=g_1\dots g_Q,\;\; g_i\in \cup_0^k K_j(r)\right\} .$$
Assume that $\mathbf F$ is a positive monotone regularly varying function of positive  index at infinity with the property that
$$ \forall\, r\ge 1,\;\;\#K(r)\ge \mathbf  F(r) .$$
Let $\mathbf F^{-1}$ be the inverse function of $\mathbf F$. Then
$$\Lambda_{2,G,\mu} \succeq 1/\mathbf F^{-1}\;\mbox{ and }  \mu^{(n)}(e)\preceq 1/\mathbf F(n),\;\; n=1,2,\dots .$$
\end{pro}
\begin{proof} This follows immediately from Propositions \ref{pro-PS1}--\ref{pro-PS3}.
\end{proof}

Let us now defined  the subset  $ \mathcal P_ {\preceq} (G,\mbox{reg})$ of $
\mathcal P(G,\mbox{reg})$ which is relevant to us because of condition (\ref{order}) and Definition \ref{def-quasinorm} (it requires that the functions $F_s$, $s\in \Sigma$, which are used to define a quasi-norm, be ordered modulo $\simeq$).
 \begin{defin}[{$\mathcal P_{\preceq} (G,\mbox{reg})$}]
 \label{def-Pregorder}  A measure $\mu=\sum_0^kp_i \mu_i$ in $\mathcal P(G,\mbox{reg})$ with associated regularly varying functions
 $\phi_i$, $1\le i\le k$,  is in
$\mathcal P_{\preceq} (G,\mbox{reg})$ if,  for each pair $i,j$ of distinct indices in $\{1,\dots,k\}$, either $\Phi_i\preceq \Phi_j$ or $\Phi_j\preceq \Phi_i$ in a neighborhood of infinity.
\end{defin}

\subsection{Application to nilpotent groups} \label{sec-nil}

The following is a corollary to Proposition \ref{pro-PS4} and Theorem \ref{th-nilgeom} (i.e., \cite[Theorems 2.10 and 3.2]{SCZ-nil}).
\begin{theo} \label{th-nil2}
Assume that $G$ is nilpotent.  Let $\mu \in \mathcal P_{\preceq}(G,\mbox{\em reg})$.
Let  $S_0$ be the support of $\mu_0$ and $S_i$ be a symmetric generating set for the subgroup $H_i$ for $1\le i\le k$.
Let
$$\Sigma =(s_1,\dots,s_m)$$
be a tuple of distinct representatives of the set $(\cup_0^kS_i)\setminus \{e\}$ under the equivalence relation $s^{-1}\sim s$, and set
$$\mathfrak F=\left\{
F_\sigma=\max\{ \Phi^{-1}_i:  i\in \{j: \sigma\in S_j\}\}: \sigma \in \Sigma\right\}.
$$
 Let $\mathbf F=\mathbf F_{G,\mathfrak F}$ be as in {\em Definition \ref{def-volnilF}}.
 Then, we have
$$\Lambda_{2,G,\mu}(v)\succeq 1/\mathbf F (v), \;\;v\ge 1,\; \mbox{ and }  \;\mu^{(n)}(e)\preceq 1/ \mathbf F(n),\;\; n=1,2\dots.$$
\end{theo}

\begin{proof} Theorem \ref{th-nilgeom} (i.e., \cite[Theorems 2.10 and 3.2]{SCZ-nil})
 shows  that the hypothesis of Proposition \ref{pro-PS4} are satisfied with $\mathbf F=\mathbf F_{G,\mathfrak F}$, which is given in Definition \ref{def-volnilF}.  The assumption that $\mu$ belongs to $\mathcal P_{\preceq}(G,\mbox{reg})$
(instead of $\mathcal P(G,\mbox{reg})$) insures that property (\ref{order}) is satisfied by the functions $F_s$, $s\in \Sigma$.\end{proof}
\begin{rem} In this theorem, whether we choose the set $S_0$ to be an arbitrary finite symmetric generating set of $G$ or the support of $\mu_0$ (as we did in the above statement) makes no difference. The reason is that $\Phi_i\preceq \Phi_0$ for each $1\le i\le k$.
\end{rem}

\begin{rem} In the case each $H_i$ is a discrete one parameter subgroup of $G$, Theorem \ref{th-nil2} is already contained in \cite{SCZ-nil}.  The case when
$k=1$, $p_0=0$, and $H_1=G$ is also known.  The theorem provides a natural extension covering these two special cases.
\end{rem}
\begin{rem} Let us emphasize here the fact that the choice of the geometry adapted to $\mu \in \mathcal P_{\preceq}(G,\mbox{reg})$ in the above theorem follows straightforwardly from the ``structure'' of the measure $\mu$ which is captured by the subgroups $H_i$ and the regularly varying functions $\phi_i$, $1\le i\le k$.
We shall see later that this is not the case when we replace the hypothesis that $G$ is nilpotent  by the hypothesis that $G$ has polynomial volume growth.
\end{rem}

\begin{exa}\label{wmoeq} We return again to the Heisenberg example (Example \ref{exa-Heisenberg}), keeping the same notation.  We let
$H_i=\<s_i\>$
(the subgroup generated by $s_i$) and $\phi_i(t)=(1+t)^{\alpha_i}$ with $\alpha_i>0$ for $1\le i\le 3$. Obviously,
the volume growth degree of each $H_i$ is $d_i=1$ for all $1\leq i\leq 3$,
 and, for $t\ge 1$,
$$\Phi_i(t) \asymp \left\{\begin{array}{cl}  t^{\alpha_i} &\mbox{ if } \alpha_i\in (0,2),\\
 t^2/ \log (1+ t)  &\mbox{ if } \alpha_i=2,\\
t^2 &\mbox{ if }  \alpha_i>2.\end{array}\right.$$
If  none of the $\alpha_i$ is equal to $2$, set $\tilde{\alpha}_i=\min\{2,\alpha_i\}$ and $\weight_i=1/\tilde{\alpha}_i$.  In this case,
$$\mathbf F (t)\asymp
\left\{\begin{array}{cl}  t^{\weight_1+\weight_2+\weight_3}  &\mbox{ if } \weight_3\ge \weight_1+\weight_2,\\
 t^{2 (\weight_1+\weight_2) }& \mbox{ if }  \weight_3\le \weight_1+\weight_2.
\end{array}\right.$$
The best way to treat the cases that include the possibility that $\alpha_i=2$ is to introduce a two-coordinate weight system and set
$$\weight_i=(\weight_{i,1},\weight_{i,2})=\left\{\begin{array}{cl} (1/\tilde{\alpha}_i, 0) & \mbox{ if } \alpha_i\neq 2, \\
 (1/2, 1/2)& \mbox{ if } \alpha_i=2.  \end{array}\right.$$
With this notation, the natural order over the functions  $F_{s_i}=\Phi_i^{-1}$ (in a neighborhood of infinity) is the same as the lexicographical order over the
weights $\weight_i=(\weight_{i,1},\weight_{i,2})$.   Furthermore, we have
$$ \mathbf F (t)\asymp
\left\{\begin{array}{cl}  t^{\weight_{1,1}+\weight_{2,1}+\weight_{3,1}} [\log (1+ t )]^{\weight_{1,2}+\weight_{2,2}+\weight_{3,2}}&\mbox{ if } \weight_3\ge \weight_1+\weight_2,\\
 t^{2 (\weight_{1,1}+\weight_{2,1})} [\log(1+t)]^{2(\weight_{1,2}+\weight_{2,2})}& \mbox{ if }  \weight_3\le \weight_1+\weight_2.
\end{array}\right.$$
The  corresponding probability of return upper bound  is already contained in \cite{SCZ-nil}. Later in this paper we  will prove the (new) matching lower bound.
\end{exa}

\begin{exa} We continue with the  Heisenberg example (Example \ref{exa-Heisenberg}), keeping the same basic notation.  Now, we let $H_1=\<s_1, s_3\>$ and $H_2=\<s_2,s_3\> $ (these are two  abelian subgroups of the Heisenberg group with $d_1=d_2=2$ as each of these subgroups is isomorphic to $\mathbb Z^2$).
We  set again  $\phi_i(t)=(1+t)^{\alpha_i}$ with $\alpha_i>0$ for $1\le i\le 2$.  The associated $\Phi_i$, $i=1,2$, are as above.  Now, we can pick $\Sigma=(s_1,s_2,s_3)$
and set
$F_{s_i}= \Phi_i^{-1}$ for $i=1,2$ and $F_{s_3}= \max\{\Phi^{-1}_1,\Phi^{-1}_2\}$.  Inspection of the construction of the weight functions on commutators
  shows   that the function $F_{s_3}$ will play no role because   $s_3=[s_1,s_2]$ and  $F_{s_3}\preceq  F_{s_1}F_{s_2}$ on a  neighborhood of infinity.
We introduce the two-coordinate weight system  $i\in \{1,2\}$,
$$\weight_i=(\weight_{i,1},\weight_{i,2})=\left\{\begin{array}{cl} (1/\tilde{\alpha}_i, 0) & \mbox{ if } \alpha_i\neq 2, \\
 (1/2, 1/2)& \mbox{ if } \alpha_i=2.  \end{array}\right.$$
The volume function $\mathbf F$ is then given by
$$ \mathbf F (t)\asymp  t^{2 (\weight_{1,1}+\weight_{2,1})} [\log(1+t)]^{2(\weight_{1,2}+\weight_{2,2})}.$$
\end{exa}

\subsection{Application to groups of polynomial volume growth} \label{sec-polvol}

It may at first be surprising that the literal generalization of Theorem \ref{th-nil2} to the case of groups of polynomial volume growth is actually incorrect.  This is because the appropriate  definition of a weight system $(\Sigma, \mathfrak F)$ associated with the data describing a probability measure
$\mu\in \mathcal P_{\preceq}(G,\mbox{reg})$ is more subtle in this case.  In order to obtain an appropriate definition, we will make use of a nilpotent approximation $N$  of the group $G$    ---    that  is, a normal nilpotent subgroup of $G$ with finite index. We will build related weight systems $(\Sigma_G,\mathfrak F_G)$ and $(\Sigma_N,\mathfrak F_N)$ that are compatible in the sense that
$$\forall\,g\in H\subset G,\;\; \|g\|_{\mathfrak F_G}\asymp \|g\|_{\mathfrak F_N}.$$

This construction will have the additional advantage to allow us to bring to bear on the polynomial volume growth case  some of the nilpotent  results of \cite{SCZ-nil}.

 Let $G$ be a group having polynomial volume growth, $\mu \in \mathcal P_{\preceq}(G,\mbox{reg})$, and $H_i,\phi_i,\Phi_i$ as in Definitions
 \ref{def-Preg} and \ref{def-PregPhi}. As $G$ has polynomial volume growth,  Gromov's theorem asserts that $G$ has  a nilpotent subgroup $N$ with finite index.  It is well known that one can choose $N$ to be a normal subgroup  (it suffices to replace $N$ by the kernel of the homomorphism  $G\mapsto \mbox{Sym}(N\backslash G)$  defined by the action of $G$ by  right-multiplication on the right-cosets $Ng$,  $g\in G$. This kernel is a subgroup of $N$ because $Ng=N$ only if $g\in N$).

 From now on, we assume that $N$ is a normal nilpotent subgroup of $G$ with finite index and we denote by $u_0,\dots, u_n$ the right-coset representatives
  so that $G$ is the disjoint union of the $Nu_i$, $0\le i\le n$, and $u_0=e$ (since $N$ is normal, these are also left-coset representatives).  We are going to use $N$ (a nilpotent approximation of $G$) to define a weighted geometry on $G$ that is suitable to study the random walk driven by $\mu$. Simultaneously, we will define a compatible geometry on $N$.

\begin{defin}[Geometry on $G$] \label{DefxiG}
Let $G$, $H_i$, $\phi_i$, $\Phi_i$, $1\le i\le k$, $\Phi_0(r)=\max\{r,r^2\}$,  and $N$  be as above.
Set $$N_i=N\cap H_i .$$
Let $S_0$ be a fixed finite symmetric generating set of $G$ and, for $1\le i\le k$,  let $S_i$ be a symmetric generating set of $N_i$.
Let $\Sigma_G$ be a set of representatives of $(\cup_0^kS_i)\setminus \{e\}$ under the equivalence relation $g\sim g^{-1}$. For each $s\in \Sigma_G$, set
\[
F_{G,s}=\max\left\{ \Phi_i ^{-1}: i\in \{0,\dots, k\}, s\in S_i\right\}.
\]
We refer to this system of weight functions  on $G$ as  $(\Sigma_G,\mathfrak F_G)$.
For  each $s\in \Sigma_G$,  fix $i=i_{G}(s)$ such that $$s\in S_i \;\;\mbox{ and }\;\; F_{G,s}=\Phi_i^{-1},$$ and set $\Sigma_G(i)=\{s\in \Sigma_G: i_G(s)=i\}$.
 \end{defin}
\begin{rem} \label{rem-index}In this definition, it is important that the set $S_i$ is a generating set of $H_i\cap N$, not of  $H_i$ itself. It is not hard to see that $N_i=H_i\cap N$ is
normal and of finite index in $H_i$. Indeed, for any subgroups $A,B$ of a group $G$, we have the index relation $[A: A\cap B]\le [A\cup B:B]$ which we apply here with $A= H_i$ and $B=N$. \end{rem}

\begin{defin}[Geometry on $N$] \label{DefxiN}
Referring to the setting and notation of Definition \ref{DefxiG}, pick a finite symmetric generating set
$\Xi_0$ in $N$ and, for each $1\le i\le k$, set
$$\Xi_i=  \cup _{\ell =0}^{m}  u_\ell\, S_i \,u_{\ell}^{-1},\;\; 1\le i\le k,$$
where the elements $u_\ell$ are the fixed right-coset representatives of $N$ in $G$.
Let $\Sigma_N$ be a set of representatives of
$(\cup_0^k  \Xi_i)\setminus\{e\}$
under the equivalence relation $g\sim g^{-1}$.  For each $s\in \Sigma_N$, set
\[
F_{N,s}=\max\left\{ \Phi_i^{-1} : i\in \{0,\dots, k\}, s\in  \Xi_i \right\} , \;\;s\in \Sigma_N.
\]
We refer to this system of weight functions  as  $(\Sigma_N,\mathfrak F_N)$.
For each $s\in \Sigma_N$, fix $i=i_N(s)$ such that
$$s\in \Xi_i \;\;\mbox{ and }\;\; F_{N,s}=\Phi_i^{-1},$$ and set $\Sigma_N(i)=\{s\in \Sigma_N: i_N(s)=i\}$.
\end{defin}
\begin{rem} In this definition, it is important that the set $\Xi_i$ is used to define $(\Sigma_N,\mathfrak F_N)$ instead of just the generating set $S_i$ of $N_i=H_i\cap N$.
\end{rem}

\begin{rem}  We are abusing notation in denoting our two weight functions systems by $\mathfrak F_G$ and $\mathfrak F_N$. Indeed,
 each of them depends on the entire data including $G$, $N$, the collection of subgroups $H_i$, the collection of functions $\phi_i$,  the choice of  $S_i$, $0\le i \le k$, $\Xi_0$, and the choice of the coset representatives $u_\ell$, $0\le \ell\le m$.
\end{rem}

One important motivation behind these two definitions is the following result.
\begin{theo}  \label{th-quasi}
Referring to the setting and notation of
{\em Definitions} $\ref{DefxiG}$  and   ${\ref{DefxiN}}$,  there are constants $0<c\le C<\infty$ such that
$$\forall g\in N,\;\; c \|g\|_{\mathfrak F_G} \le \|g\|_{\mathfrak F_N} \le C \|g\|_{\mathfrak F_G}.$$

\end{theo}
\begin{proof} (1) We first prove that for $g\in N$,  $\|g\|_{\mathfrak F_N}\le C \|g\|_{\mathfrak F_G}$.
For $g\in N$, let  $\word =\sigma^{{\eps} _1}_1\dots \sigma_p^{{\eps} _p}$, ${\eps} _i\in \{\pm 1\}$, $\sigma_i\in \Sigma_G$, be a word on the alphabet $\Sigma_G \cup \Sigma_G^{-1}$ so that $g=\word$ in $G$.   Set
$$g_0=e,\;\; g_i=\sigma^{{\eps} _1}_1\dots \sigma_i^{{\eps} _i},\;\; 0\le i\le p,$$
and write (using $g_0=e$, $g_p\in N$),
$$g=\prod_{i=1}^p \widetilde{g_{i-1}} \sigma_i^{{\eps} _i} (\widetilde{g_i})^{-1}.$$
where, for any $g\in G$,  $\tilde{g}\in \{u_0,\dots,u_m\}$ denotes the fixed coset representative of $Ng$.
Observe that, by definition,  $\widetilde{xy}=\widetilde{\widetilde{x}y}$ and that
$$\widetilde{g_i}=\widetilde{g_{i-1}\sigma_i^{{\eps} _i}}= \widetilde{\widetilde{g_{i-1}} \sigma_i^{{\eps} _i}}.$$
Setting   $j_i=j$ if $\widetilde{g_i}=u_j$, we have
$$g=\prod_{i=1}^p u_{j_{i-1}} \sigma_i^{{\eps} _i} (\widetilde{u_{j_{i-1}}\sigma^{{\eps} _i}_i })^{-1}.$$
By definition,  each factor of this product is in $N$ because $ x (\widetilde{x})^{-1}\in N $ for any $x\in G$.
If $\sigma_i\in \Sigma_G(\ell)\subset  N_\ell  \subset N$, $\ell\in\{1,\dots,k\}$, then
$$\widetilde{u_{j_{i-1}}\sigma^{{\eps} _i}_i } = u_{j_{i-1} }$$
and
$$ u_{j_{i-1}} \sigma_i^{{\eps} _i} (\widetilde{u_{j_{i-1}}\sigma^{{\eps} _i}_i })^{-1}=
u_{j_{i-1}} \sigma_i^{{\eps} _i} u_{j_{i-1}}^{-1} \in \Xi _\ell .$$

If $\sigma_i \in \Sigma_G(0)$, then we can write
$$u_{j_{i-1}} \sigma_i^{{\eps} _i} (\widetilde{u_{j_{i-1}}\sigma^{{\eps} _i}_i })^{-1}$$
as a word of uniformly bounded length at most $K$ using  elements in $\Xi_0$.

Now, we interpret this construction as providing us with a word $\word'$ over the alphabet $\Sigma_N\cup \Sigma_N^{-1}$  that represents the given element $g\in N$.
By construction,  if $\xi= u\sigma u^{-1}\in \Xi_\ell $ with $\sigma\in \Sigma_G(\ell)$, $1\le \ell\le k$, we have
$$\mbox{deg}_\xi(\word')\le \mbox{deg}_\sigma(\word)   \mbox{ and }  F_{G,\sigma}=F_{N,\xi}.$$
Further, if $\xi\in \Xi_0$, then
$$\mbox{deg}_\xi(\word')\le  K \sum_{\theta\in \Sigma_G(0)}\mbox{deg}_\theta(\word).$$
This shows that
$$\forall\,g\in N,\;\;\|g\|_{\mathfrak F_N}\le K\#\Sigma_G(0) \, \|g\|_{\mathfrak F_G}.$$

(2) We next prove that, for $g\in N$,  $\|g\|_{\mathfrak F_G}\le C \|g\|_{\mathfrak F_N}$.
For this part we rely on the main result of \cite{SCZ-nil}.   By \cite[Theorem 2.10]{SCZ-nil}, there exist an integer $Q$, a constant $C$  and a fixed sequence
$\xi_1\dots,\xi_Q \in  \Sigma_N$ such that  any $g\in N$ with  $\|g\|_{\mathfrak F_ N}= R$ can be written in the form
$$g= \prod_1^Q\xi_i^{x_i} \mbox{ with } x_i\in \mathbb Z, \;\;|x_i|\le  C F_{N,\xi_i}(R).$$
Let $\word$ be the word over the alphabet $\Sigma_N\cup\Sigma_N^{-1}$ corresponding to this product.
On the one hand, for each $i\in \{1,\dots,Q\}$ such that $\xi_i\in \Sigma_N(j)$, $j\ge 1$, there are
$u\in \{u_1,\dots,u_m\} $ and $\sigma \in S_j$ such that $\xi_i=u\sigma u^{-1}$.
Hence, for such $i$, we can write
$$\xi_i^{x_i}=u \sigma ^{x_i} u^{-1}$$
and each $u$ is a product of uniformly bounded length at most $K$ over $S_0$.
On the other hand, each $\xi_i\in \Sigma_N(0)$ can also be written as  a finite product of uniformly bounded length at most $K$ using elements in $S_0$.
Using these decompositions in the product corresponding to $\word$  gives us a word  $\word'$ representing of $g$ over the alphabet $\Sigma_G\cup \Sigma_G^{-1}$ with
$$ \mbox{deg}_\sigma(\word') \le  \mbox{deg}_{u\sigma u^{-1}}(\word)  \mbox{ if }  \xi=u\sigma u^{-1}\in \Sigma_N(j),\; u\in \{u_1,\dots,u_m\},\;1\le j\le k,$$
 and
 $$\mbox{deg}_\sigma(\word')\le  2KQ + K\sum_{i:\xi_i\in \Sigma_N(0)} |x_i|,\;\; \mbox{ if } \sigma\in  S_0.$$
Hence,
$$\forall\,g\in N,\;\; \|g\|_{\mathfrak F_G}\le 2KQ\#\Sigma_N(0) \|g\|_{\mathfrak F_N}.$$
The proof of Theorem \ref{th-quasi} is complete.  \end{proof}

\begin{defin}\label{def-FG}  Let $G$ be a group having polynomial volume growth,  $\mu \in \mathcal P_{\preceq}(G,\mbox{reg})$
and $H_i$, $\phi_i,$ $\Phi_i$ as in {Definitions
 $\ref{def-Preg}$} and $\ref{def-PregPhi}$.
Referring to the notation of   Definitions $\ref{DefxiG}$, $\ref{DefxiN}$ and $\ref{def-volnilF}$,
 set $$\mathbf F_{G,\mathfrak F_G}=\mathbf F_{N,\mathfrak {F}_N},$$
 where  $ \mathbf F_{N,\mathfrak F_N}$ is the regularly varying function associated with $(N,\Sigma_N,\mathfrak F_ N)$ by   {Definition $\ref{def-volnilF}$}.
\end{defin}

\begin{cor} \label{cor-polvol}
Referring to the setting and notation of  {\em Definitions $\ref{DefxiG}$ and $\ref{def-FG}$},  we have
$$\#\{ g\in G: \|g\|_{\mathfrak F_G}\le  R) \asymp   \mathbf F_{G,\mathfrak F_G} (R).$$
Furthermore, there exist a finite sequence $\theta_i\in \Sigma_G\setminus \Sigma_G(0)$, $1\le i\le Q$,  and a constant $C$ such that,  for any $R\ge 1$ and  any element $g$ of $G$ satisfying $\|g\|_{\mathfrak F_G} \le R$, there are elements $g_i\in G$, $0\le i\le Q$, and $x_i\in \mathbb Z$, $1\le i\le Q$, such that
$$g=g_0\prod_{i=1}^{Q} \theta_j^{x_j} g_{i} $$
 with
 $$   |x_j|\le C F_{G,\theta_j}(R), \;1\le j\le Q \; \mbox{ and }\; |g_i|^2_{S_0}\le C R,\; 0\le i\le Q. $$
Here $|a|_{S_0}$ is the word length of the element $a\in G$ over the generating set $S_0$.
\end{cor}
\begin{proof} The volume estimate follows from Theorem \ref{th-nilgeom} applied to $(N, \Sigma_N,\mathfrak F_N)$ together with
Theorem \ref{th-quasi}. Theorem \ref{th-nilgeom}  also gives
 an explicit description of the  function $\mathbf F_{N,\mathfrak F_N}=\mathbf F_{G,\mathfrak F_G}$.

Similarly, to prove the product decomposition of an element $g\in G$ stated in the corollary, we use the fact that any such $g$ can be written $hu$
with $h\in N$ and $u\in \{u_1,\dots,u_k\}$.  Obviously  $\|h\|_{\mathfrak F_G} \le C_1 \|g\|_{\mathfrak F_G}$ and, by Theorem \ref{th-quasi}, $\|h\|_{\mathfrak F_N}\le C_2 \|h\|_{\mathfrak F_G}$.  It then suffices to repeat the argument used in part (2) of the proof of Theorem \ref{th-quasi} to obtain the desired product decomposition.
\end{proof}

\begin{cor} \label{cor-polmu1}
Let $G$ be a group having polynomial volume growth,  $\mu \in \mathcal P_{\preceq}(G,\mbox{\em reg})$ and $H_i$, $\phi_i,$ $\Phi_i$ as in   {\em Definitions}
 {\em \ref{def-Preg}} and $\ref{def-PregPhi}$.    We have
$$\Lambda_{2,G,\mu}(v)\succeq 1/\mathbf F^{-1} (v), \;\;v\ge 1\;\; \mbox{ and } \;\;\mu^{(n)}(e)\preceq 1/ \mathbf F(n),\;\;n=1,2,\dots,$$
where $\mathbf F= \mathbf  F_{G,\mathfrak F_G}$.
\end{cor}
\begin{proof} Using both parts of Corollary \ref{cor-polvol}, the stated result follows from Proposition \ref{pro-PS4}.
\end{proof}

\section{Lower bounds on return probabilities} \setcounter{equation}{0} \label{sec-low}

In this section, we derive sharp lower bounds for the return
probability in the case of random walks driven by measures $\mu\in
\mathcal P_{\preceq}  (G,\mbox{reg})$, when $G$ has polynomial volume growth.
According to well-known results recalled with pointers to the
literature in Section \ref{sec-LambdaPhi},  if $\mathbf F$ is a
given regularly varying function of positive index, we have the
equivalence
$$\forall\,v\ge 1,\;\Lambda_{2,G,\mu}(v)\preceq 1/\mathbf F^{-1}(v)    \Longleftrightarrow    \forall n=1,2\dots,\;\mu^{(2n)}(e)\succeq 1/\mathbf F(n).$$

Accordingly, one of the simplest methods to obtain a lower bound on $\mu^{(2n)}(e)$ is to find a family of test functions, $\zeta_R$, $R\ge 1$,
supported on a set of volume  $\mathbf F(R)$  and such that   $$\frac{\mathcal E_{G,\mu}(\zeta_R,\zeta_R)}{\|\zeta_R\|_2^2}\preceq  1/R. $$

As $\mu$ is a convex combination of measures which are either finitely supported ($\mu_0$) or supported by a subgroup $H_i$ and associated with a regularly varying function $\phi_i$, the following two lemmas will be exactly what we need.
Lemma \ref{lem-smooth} will be used in the proof of Lemma \ref{lem-Testzeta} to control test functions constructed with the help of $\| \cdot \|_{\mathfrak F}$.

\begin{lem} \label{lem-smooth}
Let $G$ be a group of polynomial volume growth.  Let $\Sigma=(s_1,\dots,s_q)$ be a generating tuple of distinct elements  and let $\mathfrak F
=\{F_s: s\in \Sigma\}$ be a weight function system so that each $F_s$ and $F_s^{-1}\in \mathcal C^\infty([0,\infty))$, and $F_s$ is an increasing  function vanishing at $0$ and of smooth regular variation of positive index less than $1$ at infinity.
Fix $A_1$, $A_2$ and $A_3\ge 1$.  For any $g$, $h\in G$,  let $\word  \in \cup_{p=0}^\infty (\Sigma\cup \Sigma^{-1})^p$ and $R\ge 1$  with $h=\word$ in $G$,
$$\|g\|_{\mathfrak F}\le A_1R, \;\;\; \|h\|_{\mathfrak F}\le A_2R  \;\mbox{ and }\,\, F_\sigma^{-1}(\mbox{\em deg}_\sigma(\word))\le A_3R.$$Then, we have
\begin{align*}
|\|gh\|_{\mathfrak F}-\|g\|_{\mathfrak F}| &\le C_{\mathfrak
F}(A_1,A_2,A_3) \max_{\sigma\in \Sigma}\left\{ \frac{R}{F_\sigma
(R)} \mbox{\em deg}_\sigma(\word)\right\}.\end{align*}
\end{lem}

\begin{proof} Each $F^{-1}_s$ is smooth at $0$ and  of  smooth regular variation  of degree greater than $1$, so that  the derivative $[F^{-1}_s]'$ of $F^{-1}_s$
satisfies
$$\sup_{[0,T]} \left\{[F^{-1}_s]'\right\}\le C_s \frac{F^{-1}_s(T)}{ T},\quad T>0.$$
Let
$$C_{\mathfrak F}(A_1,A_3)=\max\left\{C_s \frac{F_s(t)\cdot F_s^{-1}(F_s(A_1t)+F_s(A_3t))}{t(F_s(A_1t)+F_s(A_3t))}: s\in \Sigma, t>0\right\}.$$
 If $\|gh\|_{\mathfrak F}>\|g\|_{\mathfrak F}$, let $\word_g \in \cup_{p=0}^\infty (\Sigma\cup \Sigma^{-1})^p$
be such that
$$\|g\|_{\mathfrak F}=\max_{\sigma\in \Sigma}\{ F_\sigma^{-1}(\mbox{deg}_\sigma(\word_g)\}.$$
 Set $x_\sigma=\mbox{deg}_\sigma(\word_g)$ and $y_\sigma=\mbox{deg}_\sigma(\word)$, where $\word$ is as in the statement of the lemma. We have, by Definition \ref{def-quasinorm},
\begin{align*}
|\|gh\|_{\mathfrak F}-\|g\|_{\mathfrak F} |&=\|gh\|_{\mathfrak F}-\|g\|_{\mathfrak F} \\
&\le \max_{\sigma\in \Sigma} \left\{
F_\sigma^{-1}(x_\sigma+ y_\sigma)- F_\sigma^{-1}(x_\sigma)  \right\}\\
&\le C_\sigma \frac{F_\sigma^{-1}(F_\sigma(A_1R)+F_\sigma(A_2R))}{F_\sigma(A_1R)+F_\sigma(A_2R)} y_\sigma\\
 &\le C_{\mathfrak F}(A_1,A_3)  \max_{\sigma\in \Sigma}\left\{
\left(\frac{ R }{F_\sigma(R)}\right) y_\sigma
 \right\}.\end{align*}
If, instead $\|gh\|_{\mathfrak F}<\|g\|_{\mathfrak F}$, run the same
argument with $g'=gh$ and $h'=h^{-1}$ (and $\word$ replace by
$\word^{-1}$, the formal word inverse of $\word$). This gives the
same bound with $A_1$ replaced by $A_1+A_2$.
\end{proof}

\begin{lem}  \label{lem-Testzeta}
Let $G$ be a group of polynomial volume growth.  Let
$\Sigma=(s_1,\dots,s_q)$ be a generating tuple of distinct elements,
and let $\mathfrak F =\{F_s: s\in \Sigma\}$ be a weight function
system satisfying {\em (\ref{order})} with $F_s$ and $F_s^{-1}\in
\mathcal C^1([0,\infty))$ for $s\in \Sigma$. Assume that  $F_s$ is
an increasing  function,  vanishing at $0$, and  of smooth regular
variation of positive index less than $1$ at infinity.  Let
$F_\star=\min\{F_s: s\in \Sigma\}$ be the smallest of the weight
functions. Let $H$ be a subgroup of $G$ of volume growth of index
$d$ with word length $|\cdot|$. Let $\phi$ be a positive increasing
function on $[0,\infty)$ which is also  regularly varying function
of positive index, and set
$\Phi(t)=t^2/\int_0^t\frac{2s}{\phi(s)}ds $ for $t\ge 1$. Let $\mu_H$
be a probability measure supported on $H$ and of the form
$\mu_H(h)\asymp [(1+|h|)^{d}\phi(1+|h|)]^{-1}$. Assume that there
exists $A$ such that, for any $h\in H$, there is a word
$\theta:=\theta_h\in \cup_{p=0}^\infty (\Sigma\cup \Sigma^{-1})^p$
such that $h=\theta $ in $G$ and,  for each $s\in \Sigma$,
\begin{itemize}
\item  either $\forall\,t\ge 1,$ it holds that $A F_s \circ F_\star^{-1} (\sqrt{t} )  \ge \Phi^{-1}(t)$, in which case  $\mbox{\em deg}_s(\theta) \le A |h|$,
\item or    $ \mbox{\em deg}_s(\theta) \le A. $
\end{itemize}
Under these hypotheses the function
$g\mapsto \xi_R(g)=(R-\|g\|_{\mathfrak F})_+$ satisfies
$$\frac{\mathcal E_{G,\mu_H}(\zeta_R,\zeta_R)}{\|\zeta_R\|_2^2 }\le\frac{C}{F_\star (R)^2}.$$
Also, if $\mu_0$ is a symmetric finitely supported measure,
$$\frac{\mathcal E_{G,\mu_0}(\zeta_R,\zeta_R)}{\|\zeta_R\|_2^2 }\le\frac{C}{F_\star (R)^2}.$$
\end{lem}

\begin{proof}(1) We first consider $\mu_H$.
Let $W(R)$ be the cardinality of $\{g: \|g\|_{\mathfrak F} \le R\}$.  Since $\zeta_R$ is  at least $R/2$ over $\{\|g\|_{\mathfrak F}\le R/2\}$,
we have $\|\zeta_R\|_2^2\asymp R^2 W(R)$.
We now bound
$$\mathcal E_{G,\mu_H}(\zeta_R,\zeta_R)=\frac{1}{2}\sum _{g,h}|\zeta_R(gh)-\zeta_R(g)|^2\mu_H(h)$$ from above.  Let
$$\Omega=\{(g,h)\in G\times H: \zeta_R(gh)+\zeta_R(g)>0\}.$$
Obviously, the sum defining $\mathcal E_{G,\mu}(\zeta_R,\zeta_R)$ can be restricted to $\Omega$ and, for any given $h$,
$$\#\{g\in G: (g,h)\in \Omega\}\le 2 W(R).$$
We split this sum into two parts depending on
whether or not $|h|\ge \rho$ where $\rho $ will be chosen later, and write
\begin{align*}
\sum_{(g,h)\in \Omega: |h|\ge \rho}|\zeta_R(gh)-\zeta_R(g)|^2\mu_H(h) &\le   2R^2 W(R) \left(\sum_{|h|\ge \rho}\mu_H(h)\right)\\
&\le  CR^2W(R)/\phi(\rho)\le CR^2W(R)/\Phi(\rho),
\end{align*}
where we have used the simple fact that $\phi\ge \Phi$ in the last inequality, see Remark \ref{rem-Phi}. On the other hand, consider
pairs  $(g,h)\in \Omega$ with $|h|< \rho$ with
$\rho=\Phi^{-1}\left(F_\star(R)^2\right)$.  This choice of $\rho$
ensures that, for any $\sigma\in \Sigma$ such that $A F_\sigma\circ
F_\star^{-1}(\sqrt{t})\ge \Phi^{-1} (t)$   for all $t>1$,  we have
$\mbox{deg}_\sigma(\theta)\le  A \rho$. Thus,
$$F^{-1}_\sigma(\mbox{deg}_\sigma(\theta))\le A' F_\star^{-1}\left(\sqrt{\Phi (\rho)}\right)\le A'R$$
and
$$ \frac{R}{F_\sigma(R)}\le A'\frac{R}{\Phi^{-1}(F_\star(R)^2)}.$$
Since $(g,h)\in \Omega$, it follows that all $\|h\|_{\mathfrak
F}$, $\|g\|_{\mathfrak F}$ and $\|gh\|_{\mathfrak F}$ are smaller
than $A''R$ for some constant $A''>0$. Combining all the estimates above with the assumptions of the lemma (in particular, $F_\star$ is the smallest of the weight functions) and Lemma \ref{lem-smooth} yields that
$$
|\zeta_R(gh)-\zeta_R(g)| \le |\|gh\|_{\mathfrak F}-\|g\|_{\mathfrak F}| \le C_1\left(\frac{R}{F_\star(R)} + \frac{R}{\Phi^{-1}(F_\star(R)^2)}|h|\right).
$$
It follows that
\begin{align*}
\sum_{(g,h)\in \Omega:|h| <\rho}|\zeta_R(gh)&-\zeta_R(g)|^2\mu_H(h)\\ &\le  4C_1^2W(R)\left(\sum_{|h|<\rho} \left(   \frac{R^2}{F_\star(R)^2} + \frac{R^2}{[\Phi^{-1}(F_\star(R)^2)]^2}|h|^2            \right) \mu_H(h)\right)\\
&\le  4C_1^2R^2W(R)\left(\frac{1}{F_\star(R)^2} + \frac{1}{[\Phi^{-1}(F_\star(R)^2)]^2}\sum_{|h|\le \rho}|h|^2\mu_H(h)\right) \\
&\le   C_2 R^2 W(R)  \left(\frac{1}{F_\star(R)^2} + \frac{ \rho^2 }{\Phi(\rho) [\Phi^{-1}(F_\star(R)^2)]^2} \right)\\
&\le  \frac{C_2 R^2W(R)}{F_\star(R)^2}.
 \end{align*}

(2) We next consider $\mu_0$. In the case of $\mu_0$, for
$(g,h)\in \Omega$ , we have $|h|\le C$ (because $h$ is in the
support of $\mu_0$), and both $\|g\|_{\mathfrak F}$ and
$\|gh\|_{\mathfrak F}$ are smaller than $(1+C)R $ (because $(g,h)\in
\Omega$, $h$ is in the support of $\mu_0$ and $R\ge 1$) for some constant $C>0$. Thus, it
follows from Lemma \ref{lem-smooth} that
$$
|\zeta_R(gh)-\zeta_R(g)| \le |\|gh\|_{\mathfrak F}-\|g\|_{\mathfrak F}| \le C_1\left(\frac{R}{F_\star(R)} \right).
$$
Using this in computing $\mathcal E_{G,\mu_0}(\zeta_R,\zeta_R)$ yields
\begin{align*}
\sum_{(g,h)\in \Omega}|\zeta_R(gh)-\zeta_R(g)|^2\mu_0(h) &\le  4C_1^2W(R) \sum_h \frac{R^2}{F_\star(R)^2} \mu_0(h)\\
&\le \frac{2C_1 R^2W(R)}{F_\star(R)^2}.
 \end{align*} Combining all the estimates above yields the desired conclusion.
\end{proof}

\begin{theo} \label{th-polmu}
Let $G$ be a group having polynomial volume growth and $\mu $ be in
$\mathcal P_{\preceq}(G,\mbox{\em reg})$. Let  $H_i,\phi_i,\Phi_i$ as in {\em
Definitions
 \ref{def-Preg}} and $\ref{def-PregPhi}$.
Referring to the notation of {\em Definitions \ref{DefxiG}}, $\ref{DefxiN}$ and  $\ref{def-FG}$,
 set $\mathbf F=\mathbf F_{G,\mathfrak F_G}$.
 Then, we have
$$\Lambda_{2,G,\mu}(v)\asymp 1/\mathbf F^{-1} (v), \;\;\;\mu^{(n)}(e)\asymp 1/ \mathbf F(n).$$
\end{theo}
\begin{proof}
By the results mentioned in Section \ref{sec-LambdaPhi}, it suffices
to prove that $$\Lambda_{2,G,\mu} (v)\asymp 1/\mathbf F^{-1}(v).$$
Corollary \ref{cor-polmu1} provides the lower bound
$\Lambda_{2,G,\mu}(v)\succeq 1/\mathbf F^{-1}(v)$ so it suffices to
prove the upper bound.  Consider the weight function system
associated with $(\Sigma_G,\mathfrak F_G)$, where $\mathfrak
F_G=\{F_s, s\in \Sigma_G\}$ is as in Definition  \ref{DefxiG}. Let
$\weight (s)$ be the positive index of the increasing regularly varying
function $F_s$.  By Proposition \ref{pro-variation}, there are
functions $F_{1,s}\in \mathcal C^1([0,\infty))$ for all $s \in
\Sigma_G $, positive, increasing and of smooth regular variation of
index strictly less than $1$ such that $F^{-1}_{1,s}\in \mathcal
C^1([0,\infty))$ and $F_{1,s} \asymp F_s\circ F^{-1}$, where
$F(t)=(1+t)^{\weight^*}-1$ and $\weight^*$ is chosen so that $\weight^*> \max\{\weight (s):
s\in \Sigma_G\}$.  
These functions satisfy (\ref{Fder}).   Let
$(\Sigma_G,\mathfrak F_{G,1})$ be the associated weight function
system and recall that, by construction, 
$$\|g\|_{\mathfrak F_G}\asymp \|g\|_{\mathfrak F_{G,1}}^{1/\weight^*}.$$

Let $\zeta_R(g)=(R-\|g\|_{\mathfrak F_{G,1}})_+$.
For each $H_i$, $\mu_i$, $1\le i\le k$, and $s\in  S_i$  (see Definition \ref{DefxiG}), we  have
$$ F_{1,s}\circ F^{-1}_{1,\star} (\sqrt{t}) \asymp  F_{s}\circ F^{-1}_{\star} (\sqrt{t}) \ge (1/A) F_s(t) \ge (1/A) \Phi^{-1}_i(t).$$
Here we have used the simple fact  that, because of the definition of the system $\mathfrak F_G$ based on the function $\Phi_i$,
$F_\star (t) \asymp  \min\{t,\sqrt{t}\}$  (See Remark \ref{rem-Phi} and the definition of $\Phi_0$).
Further, since $N_i=H_i\cap N$ is of finite index in $H_i$ (See Remark \ref{rem-index}). Hence, for any $h\in H_i$, there is a word
$\theta:=\theta_h\in  \cup_{p=0}^\infty (\Sigma _G\cup \Sigma_G^{-1})^p$ such that $h=\theta $ in $G$ and,  for each $s\in \Sigma_G\setminus S_i$,
$ \mbox{deg}_s(\theta) \le A. $

It follows that we can apply Lemma \ref{lem-Testzeta}  to $\mu_i$
and $\mathfrak F_{1,G}$. This gives that for $R\ge 1$, 
$$\frac{\mathcal E_{G,\mu_i}(\zeta_R,\zeta_R)}{\|\zeta_R\|_2^2 }\le\frac{C}{F_{1,\star} (R)^2}$$
with $F_{1,\star}(t) \asymp {t}^{1/(2\weight ^*)}$ for all $t\ge1$; that is,
$$\frac{\mathcal E_{G,\mu_i}(\zeta_R,\zeta_R)}{\|\zeta_R\|_2^2 }\le\frac{C}{R^{1/\weight^*}}
\quad \hbox{for every } R\ge 1.$$
This holds for all $i=1,\dots,k$, and also for $\mu_0$. Hence it also holds for the convex combination
 $\mu\in \mathcal P_{\preceq}(G,\mbox{reg})$;
that is,
$$\frac{\mathcal E_{G,\mu}(\zeta_R,\zeta_R)}{\|\zeta_R\|_2^2}\le\frac{C}{R^{1/\weight^*}}
\quad \hbox{for every } R\ge 1.$$
Now,  $\zeta_R$ is supported in  $\{\|g\|_{\mathfrak F_{G,1}}<R\}$
which has volume
$$\# \{\|g\|_{\mathfrak F_{G,1}}\le R\} \asymp \#\{\|g\|_{\mathfrak F_G}^{\weight ^*}\le R\} \asymp  \mathbf F(R^{1/\weight^*}),$$ thanks to Corollary \ref{cor-polvol}.
It follows that
$$\forall \, v\ge 1,\;\;\Lambda_{2,G,\mu}(v)\preceq \frac{1}{\mathbf F^{-1}(v)}.$$
The proof is complete.\end{proof}

\begin{rem} Even in the simplest case when $G$ is nilpotent, each $H_i$ is a one parameter subgroup
$H_i=\<s_i\>\subseteq G$
and the measure $\mu_i$  satisfies $\mu_i(h) \asymp
(1+|h|)^{-1-\alpha_i}$ with $\alpha_i>0$, $h=s_i^n$ and $n\in
\mathbb Z$, (i.e., the basic case studied in \cite{SCZ-nil}),
Theorem \ref{th-polmu} provides new results, since \cite{SCZ-nil}
gives complete two sided bounds only  in the case when all
$\alpha_i$ are in  the interval $(0,2)$ or all $\alpha_i$ are equal
to $2$.
\end{rem}

\begin{exa} We return again to the Heisenberg example (Example \ref{exa-Heisenberg}), keeping the same notation.
We let $H_i=\<s_i\>$ (the subgroup generated by $s_i$) and
$\phi_i(t)=(1+t)^{\alpha_i}$ with $\alpha_i>0$ for $1\le i\le 3$.
Let $\mu=\sum_{i=1}^3\mu_i$. Obviously, $d_i=1$ and, for $t\ge 1$,
$$\Phi_i(t) \asymp \left\{\begin{array}{cl}  t^{\alpha_i} &\mbox{ if } \alpha_i\in (0,2),\\
t^2 /\log (1+ t)   &\mbox{ if } \alpha_i=2,\\
t^2 &\mbox{ if }  \alpha_i>2.\end{array}\right.$$
Introduce the two-coordinate weight system
$$\weight_i=(\weight_{i,1},\weight_{i,2})=\left\{\begin{array}{cl} (1/\tilde{\alpha}_i, 0) & \mbox{ if } \alpha_i\neq 2, \\
 (1/2, 1/2)& \mbox{ if } \alpha_i=2.  \end{array}\right.$$
Note that the natural order over the functions  $F_{s_i}=\Phi_i^{-1}$ (in a neighborhood of infinity) is the same as the lexicographical order over the
weights $\weight_i=(\weight_{i,1},\weight_{i,2})$ and that
$$ \mathbf F (t)\asymp
\left\{\begin{array}{cl}  t^{\weight_{1,1}+\weight_{2,1}+\weight_{3,1}} [\log (1+ t )]^{\weight_{1,2}+\weight_{2,2}+\weight_{3,2}}&\mbox{ if } \weight_3\ge \weight_1+\weight_2,\\
 t^{2 (\weight_{1,1}+\weight_{2,1})} [\log(1+t)]^{2(\weight_{1,2}+\weight_{2,2})}& \mbox{ if }  \weight_3\le \weight_1+\weight_2.
\end{array}\right.$$
Theorem \ref{th-polmu} tells us that
$$\Lambda_{2,\mathbb H(3,\mathbb Z),\mu}(v) \asymp \frac{1}{\mathbf F^{-1}(v)},\;\;\;\mu^{(n)}(e)\asymp \frac{1}{\mathbf F(n)}.$$
Next consider the case when $\phi_i(t)=(1+t)^{2}[\log
(2+t)]^{\beta_i}$ with $\beta_i\in \mathbb R$ for $i=1,2$, and
$\phi_3(t)=(1+t)[\log (2+t)]^{\beta_3}$ with $\beta_3\in \mathbb R$.
In this case, for $t\ge 1$ and $i=1,2$,
$$\Phi_i(t) \asymp \left\{\begin{array}{cl} t^2 \log (2+t)]^{\beta_i -1} & \mbox{ if } \beta_i < 1,\\
t^2 /\log\log(4+t)  & \mbox{ if } \beta_i=1,\\
t^2 &\mbox{ if }  \beta_i>1,\end{array}\right.$$
and
$$ \forall\,t\ge 1,\;\;  \Phi_3(t)\asymp t [\log (2+t)]^{\beta_3}.$$
We need to compare  $\Phi^{-1}_1\Phi^{-1}_2$ to $\Phi^{-1}_3$ over $(1,\infty)$.

Assume first that $\beta_1,\beta_2\in (-\infty,1)$ so that
$$\Phi_1^{-1}(t)\Phi_2^{-1}(t)\asymp  t [\log(2+t)]^{ -(\beta_1+\beta_2-2 )/2}$$ and $ \Phi^{-1}_3(t) \asymp  t[\log (2+t)]^{-\beta_3}$.
When $\frac{1}{2}(\beta_1+\beta_2-2)> \beta_3$, the function
$\Phi_3^{-1}$ dominates and the volume function  $\mathbf F$
associated with the weight function system $F_{s_i}=\Phi_i^{-1}$ is
the product $ \mathbf F=\Phi_1^{-1}\Phi_2^{-1}\Phi^{-1}_3$.   If, instead, $\frac{1}{2}(\beta_1+\beta_2-2)\le  \beta_3$,
then we have $\mathbf F= (\Phi^{-1}_1\Phi_2^{-1} )^2$.

In the case when
$\beta_1,\beta_2\ge 1$,  one can check that
$$\mathbf F\asymp \left\{ \begin{array}{cl} \Phi_1^{-1}\Phi_2^{-1}\Phi^{-1}_3  & \mbox{ if } \beta_3 < 0,
 \\
  (\Phi^{-1}_1\Phi_2^{-1})^2 & \mbox{ if } \beta_3\ge0.
  \end{array}\right.$$
  In all cases (including those not discussed explicitly above),
  $$\mathbf F= \Phi_1^{-1}\Phi_{2}^{-1}\max\{ \Phi_3^{-1},\Phi_1^{-1}\Phi_2^{-1}\}.$$
\end{exa}

The following examples illustrate some of the subtleties related to the treatment of groups of polynomial growth that are not nilpotent.

\begin{exa}[Infinite dihedral group]
Recall that the infinite dihedral group $\mathbf D$ can be presented as $\<u,v: u^2,v^2\>$.  This means it is the quotient of the free group $\mathbf F(\mathbf u,
\mathbf v)$
with two generators by the normal subgroup of $\mathbf F(\mathbf u,\mathbf v)$ generated  by $\mathbf u^2,\mathbf v^2$ (and all the conjugates of these).  Obviously,  the images  $u,v$ of $\mathbf u, \mathbf v$ in the quotient, $\mathbf D$,  satisfy $u^2=v^2=e$.
Since $\{u,v\}$ generates $\mathbf D$, any element $g$ can be written uniquely as $g=(uv)^n u^{{\eps} }$ with $n\in \mathbb Z$ and ${\eps} \in \{0,1\}$ and also as $g=(uv)^m v^{\eta}$ with $m\in \mathbb Z$ and $\eta\in \{0,1\}$.   For instance   $vu= (uv)^{-1}$.  The Cayley graph  looks like this:

\begin{figure}[h]
\begin{center}\caption{Dihedral  group:  $u$ in blue, $v$ in red.}\label{D1}
\begin{picture}(300,50)(0,-10)

\multiput(254,0)(-15,0){15}{\circle*{3}}

{\color{blue}
\multiput(251,0)(-30,0){8}{\line(-1,0){10}}
}

{\color{red}
\multiput(263,0)(-30,0){8}{\line(-1,0){10}}
}

\end{picture}\end{center}\end{figure}

This group is not nilpotent  because the commutator of  $(uv)^n$ by $u$ is $(uv)^{2n}$. Let $z=uv$ and $Z=\<z\>$, the subgroup generated by $z$. Since $z$ has infinite order, this is a copy of $\mathbb Z$. It is also a normal subgroup with quotient the group with two elements $\{0,1\}$.  From what we said before, any element $g$ can be written uniquely $g=t^n u^{{\eps} }$ and this gives a description of $\mathbf D$ as the semi-direct product
$$
\mathbf D=Z\rtimes \<u\>.
$$

Obviously (from the picture) $\mathbf D$ has volume growth of degree $1$. The subgroup $Z$ is abelian  (hence, nilpotent) with finite index.
Even so this example is relatively trivial, it already illustrates significant differences with the nilpotent case.

For instance, it is easy to see that Theorem 3.9 does not apply as stated to this group.
 Indeed, take  $\mu=(1/2)(\mu_1+\mu_2)\in \mathcal P(\mathbf D,\mbox{reg})$ where
$\mu_1$, $\mu_2$ are associated to $\phi_1(s) =s^{\alpha_1}$, $\phi_2(s)=s^{\alpha_2}$ on the subgroups $\langle u\rangle,\langle v\rangle$, respectively.
Assume that $0< \alpha_1\le \alpha_2<2$.

On the one hand, since $u,v$ have order $2$, it is obvious that this is just a weighted version of simple random walk on $\Delta$ with generators  $\{u,v\}$.
On the other hand, we can consider the norm $\|\cdot\|$ associated with the system $\{u,v\}$,  $F_u(s)=(1+s)^{1/\alpha_1}-1$, $F_u(v)=(1+s)^{1/\alpha_2}-1$.
That would be the norm used in Theorem 3.9 if $\Delta$ was nilpotent. Obviously,  the element $g=(uv)^nu^{{\eps} }$ as norm 
$\|g\|\asymp (|{\eps} |+|n|)^{\alpha_2}.$  If they where applicable, Theorems 2.10 and 3.9 would say $\mu^{(2n)}(e)\asymp n^{-1/\alpha_2}$ whereas, obviously, $\mu^{(2n)}(e)\asymp n^{-1/2}$.

Why is it the case that such situation does not appear in the nilpotent case treated in Theorems 2.10 and  3.9?   The reason is algebraic, namely, in any finitely  generated nilpotent group, the torsion elements form a {\em finite} normal subgroup and torsion elements  cannot play a significant role in computing the type of length considered in Theorem 2.10.  Obviously, the dihedral group $\mathbf D$ is very different from this view point.

\end{exa}

\begin{exa}  The goal of this example is to  construct a multi-dimensional version of the dihedral group above which will illustrate how our results apply to groups of polynomial growth that are not nilpotent. Let  $\Delta$ be generated by
$s,s',t,t'$ which are all of order two (involutions) and which satisfy the following commutation relations
$$ tt'=t't,\; ss'=s's,\; s't'=t's' ,\; stst'=st'st,\; sts't=s'tst,\; st's't=s'tst'.$$
Technically, this means that $\Delta$ is the quotient of the free group generated by four letters $s,s',t,t'$  by the normal subgroup generated by the indicated relations which are all commutation relations between two elements:  $t$ and $t'$ commute ($[t,t']=1$), $s$ and $s'$ commute ($[s,s']=1$), $s'$ and $t'$ commute ($[s',t']=1$), $st$ and $st'$ commutes ($[st,st']=1$), $st$ and $s't$ commute ($[st,s't]=1$), and $st'$ and $s't$ commutes ($[st',s't]=1$).
The next lemma gives a concrete description of this group.
But in the present case  it is possible to obtain a good picture of what happens.

There are three homomorphisms onto the Dihedral group $\mathbf D$, $\psi_1,\psi_2,\psi_3$, obtained by sending
the pair $(s',t)$ (resp. $(s,t)$, $(s,t')$) to $(u,v)$ and all other generators to the identity.   For instance $\psi_1$ is obtained by looking at the homomorphism
$\psi$ from the free group $\mathbf F_4$ (with four generators $\mathbf t,\mathbf t,\mathbf s,\mathbf s'$)  onto $\mathbf D$ which send $\mathbf s'$ to $u$, $\mathbf t$ to $v$, and the other generators to the identity. By construction, the kernel of this homomorphism  contains  the defining kernel  $N_\Delta$ of the group  $\Delta=\mathbf F_4/N_\Delta$.  It follows that $\psi$ descends to an surjective homomorphism  $\psi_1:\Delta \longrightarrow \mathbf D$.
We get one such homomorphism for each pair $(s',t)$, $(s,t)$ and $(s,t')$.  This proves that the commuting elements $st,st'$ and $s't$
each have infinite order.  There are also three copies of the dihedral group
$\<s',t\>$, $\<s,t\> $ and $\<s,t'\>$ in $\Delta$.

\begin{lem} Any element $g$ of $\Delta$ can be represented uniquely in the form
$$g= (s't)^{n_1}(st)^{n_2}(st')^{n_3}s^{\eps} ,\;\; {\eps}  \in \{0,1\},\;\;n_1,n_2,n_3\in \mathbb Z$$
and  the commuting elements  $s't,st,st'$ generate a copy of  $\mathbb Z^3$ as a subgroup of $\Delta$.
\end{lem}
\begin{proof}
Existence of $(n_1,n_2,n_3,{\eps} )$ can be proved by induction on the length of words on the generators $s,s',t,t'$. For this, it suffices to check
that if $g$ has the desired form, then we can also write $gs,gs',gt,gt'$ in the desired form. This is easy for  $s$ (!) and $t'$.  If ${\eps} =1$, it is also easy to move $t$ through since $st$ commutes with $st'$.  If ${\eps} =0$, move $t$ by writing $t=(ts)s $.  Finally, $s'$ commutes with $s$ and $t'$  so that
\begin{align*}
gs' &=(s't)^{n_1}(st)^{n_2}s' (st')^{n_3}s^{{\eps} } = (s't)^{n_1}(st)^{n_2}(s't) (ts)s (st')^{n_3}s^{{\eps} }\\
&= (s't)^{n_1+1}(st)^{n_2-1}(st')^{n'_3}s^{{\eps} '},
\end{align*}
where, on the right, we have used the fact that $ s(st')^{n_3}s^{\eps} =(st')^{n'_3}s^{{\eps} '}$  since any element in $<s,t'>$ can be written in that form (of course, one can also compute explicitly $(n'_3,{\eps} ')$ as a function of $(n_3,{\eps} )$).

To prove uniqueness,  assume an element $g$ in $\Delta$ can be written in two different ways $(n_1,n_2,n_3,{\eps} )$ and $(n_1',n_2',n_3',{\eps} ')$ or, equivalently,
$$s^{{\eps} '} (s't)^{n_1-n_1'}(st)^{n_2-n_2'}(st')^{n_3-n_3'}s^{\eps} =1   \mbox{ in } \Delta.$$
We consider the images of this by $\psi_1,\psi_2,\psi_3$.   It will be useful to note that in $\mathbf D=\langle u,v\rangle$,
$u^\eta (uv)^k u^{\eta'}=1$ implies $k=0$ and also $v^\eta (uv)^k v^{\eta'}=1$ implies $k=0$.
Taking the image by $\psi _1$ gives
$$ v^{n'_2}(uv)^{n_1-n'_1}v^{n_2}= 1 \mbox{ in the dihedral group}.  $$
This implies $n_1=n'_1$.  Taking the image by $\psi_2$ gives
$$ u^{{\eps} '}(uv)^{n_2-n'_2} u^{n_3-n'_3+{\eps} }=1$$
in the dihedral group which  implies $n_2=n'_2$. Finally, taking the image by $\psi_3$ gives
$$u^{{\eps} ' +n_2-n'_2}(uv)^{n_3-n'_3}u^{{\eps} }=1$$
in the dihedral group which implies $n_3=n'_3$. Obviously, we also must  have ${\eps} ={\eps} '$.
\end{proof}

Next we show that this writing  is (almost) minimizing the degree of each of the generator $s,t,s',t'$ in a word representing $g$.
For this, we first need to observe that  for each element $w$ in the dihedral group $\<u,v: u^2,v^2\>$, there exists a smallest $n\ge 0$ such that
$w= \word_w$ with $\word_w=(uv)^{n}$ or $\word_w=(uv)^{-n}$ or $\word_w=(uv)^n u$ or $\word_w=(uv)^{-n} v$. Moreover, if $\word$ is any word over $\{u,v\}$ such that $w=\word$ then we have
\begin{equation}\label{degD}
\mbox{deg}_x(\word)\ge \mbox{deg}_x(\word_w),\;\;x=u,v.
\end{equation}
This fact is nothing difficult. Each element of $\mathbf D$  is, uniquely and   minimally,   a product iterating between $u$ and $v$. The four cases mentioned above are exactly the cases when $w$  starts with $u$ and finishes with $v$, starts with $v$ and finishes with $u$, starts with $u$ and finishes with $u$, and starts with $v$ and finishes with $v$.

Now, for any word $\word$ over $\{s,t,s',t'\}$ representing $g=(n_1,n_2,n_3,{\eps} )$ and $i=1,2,3$, let $\bar{\psi} _i(\word)$ be the word obtained in an obvious way by cancelling those letters that are sent to the identity by $\psi_i$.  Note that the word $\bar{\psi}_i(\word)$ represents the element
$\psi_i(g)$ in the dihedral group.  For $x=s,t,s',t'$,
we have $$\mbox{deg}_{x}(\word)\ge \mbox{deg}_{\psi_i(x)}(\bar{\psi}_i(\word))\ge |n_i| - 1$$
because, for instance,  $\psi_2(g)= t^{n_1}(st)^{n_2}s^{n_3+{\eps} }=t^{\eta_1}(st)^{n_2}s^{\eta_2}$ with $\eta_1,\eta_2\in \{0,1\}$. Compare with (\ref{degD})
to obtain the desired result.
Note that the previous inequality  is optimal since in order to write $t$ in the form $(n_1,n_2,n_3,{\eps} )$ we have to write $t=tss=(0,-1,0,1)$.

Let
$$\theta_1=s't,\;\theta_2= st,\; \theta_3=st'$$ be the generators of the subgroup
 $\mathbb Z^3= \<\theta_1,\theta_2,\theta_3 \>$ in $\Delta$.
Let
$$
H_1=\<s',t\>,\; H_2=\<s,t\>,\; H_3=\<s,t'\>
$$  be the three copies of the dihedral group. Pick
$$\alpha_1,\alpha_2,\alpha_3 \in (0,2) \mbox{ and } \phi_i(t)=(1+t)^{\alpha_i}.$$
 Let $\mu_i$ be a measure supported on $H_i$ with $\mu_i(h)\asymp  [(1+|h|_i)\phi_i(1+|h|_i)]^{-1}$ and
 $$\mu=(1/3)\sum_1^2\mu_i   \in \mathcal P(\Delta,\mbox{reg}).$$

If $\Delta$ were
nilpotent (it is not!), we would obtain an appropriate geometry by  setting  $\Sigma=\{s,s',t,t'\}$, $\weight (\sigma)=\max \{ 1/\alpha_i : i\in \{j: \sigma\in H_i\}\}$, and $F_\sigma(t)=(1+t)^{\weight (\sigma)}-1$.  (See Definition \ref{def-quasinorm}).
Here we have picked a generating set for each $H_i$, assigned the weight $1/\alpha_i$ to the generators coming from $H_i$, reduced to $\Sigma$ to avoid repetition and picked the highest available weight for each $\sigma\in \Sigma$.
We can compute the associated geometry.  We have
$$\weight (s)= \max \{1/\alpha_2,1/\alpha_3\}, \;\;\weight (t)=\max \{1/\alpha_1,1/\alpha_2\},$$
and  $\weight (s')=1/\alpha_1$, $\weight (t')=1/\alpha_3$.   Set
\begin{align*}
\weight_1&=\min\{\weight (s'),\weight (t)\}=1/\alpha_1,\\
\weight_2&=\min\{ \weight (s),\weight (t)\}=\min \{\max \{1/\alpha_2,1/\alpha_3\},\max \{1/\alpha_1,1/\alpha_2\}\},\\
\weight_3&=\min\{\weight (s),\weight (t')\}=1/\alpha_3. \end{align*}
Note that $\weight_2=1/\alpha_2$ if and only if $\alpha_2\le \max \{\alpha_1,\alpha_3\}.$
By inspection, using the lower bound on degrees derived above in terms of the $|n_i|$, we find that if $g$ is represented by $(n_1,n_2,n_3,{\eps} )$, we have
$$  \|g\|_{\mathfrak F}\asymp \max\{ |n_1|^{1/\weight_1},|n_2|^{1/\weight_2},|n_3|^{1/\weight_3},|{\eps} |\}.$$
Accordingly, $$\#\{\|g\|_{\mathfrak F}\le R\}\asymp R^{\weight_1+\weight_2+\weight_3}.$$
However, this quasi-norm is not appropriate to study the measure $\mu$.

Now, let us follow carefully Definition \ref{DefxiG} in order to obtain a quasi-norm that is adapted to the study of the random walk driven by $\mu$.
As a normal nilpotent subgroup $N$ with finite index in $\Delta$, we take
$$
N=\mathbb Z^3= \<s't,st,st'\>=\<\theta_1,\theta_2,\theta_3\>.
$$
We obtain  $N_i=N\cap H_i= \<\theta_i\>$ with generating set $S_i=\{\theta_i\}$. We pick a generating set $S_0=\{s,t,s',t'\}$ of $\Delta$. We form
$\Sigma_\Delta=\{\theta_1,\theta_2,\theta_3,,s,t,s', t'\}$ equipped with the weight functions system $\mathfrak F_\Delta$,
$$F_{\theta_i}(r)=(1+r)^{1/\alpha_i}-1,\;\; i=1,2,3,$$
and $$ F_{s}(r)=F_{s'}(r)=F_{t}(r)=F_{t'}(r)=\min\{r,r^{1/2}\}.$$
In this system, we can compute (this takes a bit of inspection) that if $g$ is represented by $(n_1,n_2,n_3,{\eps} )$
then
$$\|g\|_{\mathfrak F_D}\asymp\{ |n_1|^{1/\alpha_1},|n_2|^{1/\alpha_2},|n_3|^{1/\alpha_3}, |{\eps} |\}$$
and $\#\{\|g\|_{\mathfrak F_\Delta} \le R\}\asymp R^{d}$ with $d=  \sum_1^31/\alpha_i$.  Regarding the random walk driven by $\mu$, we obtain
$$\Lambda_{2,\Delta,\mu}(v)\asymp  v^{-1/d},\;\;\mu^{(n)}(e)\asymp n^{- d}.$$

\end{exa}

\begin{exa}
Consider  semi-direct product $G=\mathbb Z\ltimes _\rho\mathbb Z^2$ where $k\in \mathbb Z$ acts on
$\mathbb Z^2$ by $\rho^k$ where $\rho (n_1,n_2)= (-n_2,n_1)$ (i.e., $\rho$ is the counter-clockwise rotation of  $90{\degree}$).  Concretely, the
group elements are elements $(k,n_1,n_2)$  of $\mathbb Z^3$ and multiplication is given by
$$(k,n_1,n_2)\cdot (k',n'_1,n'_2)=(k+k', m_1,m_2), $$
where
$$m=(m_1,m_2)=n+\rho^k(n'), \;n=(n_1,n_2),\;n'=(n'_1,n'_2).$$

This group is not nilpotent, but it is
of polynomial volume growth of degree $3$ with normal abelian subgroup $N=4\mathbb Z\times \mathbb Z^2$ because $\rho^4=\mbox{id}$. Let $s,v_1,v_2$
be canonical generators of $\mathbb Z$ and $\mathbb Z^2$.
Consider the subgroups
$$
H_1= \langle 4s,v_1\rangle, \quad  H_2= \langle 4s,v_2 \rangle ,
$$
and set $\phi(t)=(1+t)^{\alpha_i}$, $\alpha_i\in (0,2)$, $i=1,2$.   Let $$\mu=(1/3)\sum_0^3\mu_i$$ with
$\mu_0=\frac{1}{2}\mathbf 1_{\{s,s^{-1}\}}$ and $$\mu_i ((k,n_1,n_2)\asymp (1+|k|+|n_i|)^{-2-\alpha_i}\mathbf 1_{H_i}, \,i=1,2.$$

We now describe the geometries $\mathfrak F_G$ and $\mathfrak F_N$ which are such that
$$\|g\|_{\mathfrak F_G}\asymp \|g\|_{\mathfrak F_N} \mbox{ for all }g\in N$$ and are  compatible with the random walk driven by $\mu$.

For $S_i$ (a generating set of $N\cap H_i=H_i$), we take $S_i=\{4s,v_i\}$. For $S_0$, a generating set of $G$, we take $S_0=\{s,v_1,v_2\}$.
We obtain  $\Sigma_G=\{s,4s,v_1,v_2\}$ with associated $F_\sigma$ functions
$F_{s}(t)=\min\{t,t^{1/2}\}\asymp (1+t)^{1/2}-1$ and  $$ F_{4s}(t)=(1+t)^{\max\{1/\alpha_1,1/\alpha_2\}}-1,\; f_{v_i}(t)=(1+t)^{1/\alpha_i}-1.$$ Set $\alpha=\min\{\alpha_1,\alpha_2\}$.
By inspection, for $g=(k,n_1,n_2)$,
$$\|(k,n_1,n_2)\|_{\mathfrak F_G} \asymp \max\{ |k|^{\alpha},|n_1|^{\alpha}, |n_2|^{\alpha}\}.$$
The same power $\alpha$ appears for both  $n_1,n_2$  because of the use of the rotation $\rho$.

Regarding the associated geometry on $N$, we take  $\Xi_0=\{4s,v_1,v_2)$,   and   $$\Xi_i= \{4s, v_i, \rho(v_i)=\pm v_j, \rho^2(v_i)= -v_i,\rho^3(v_i)=\mp v_j\}, \; i\ne j, \; i,j\in\{1,2\}.$$
This gives $\Sigma_N=\{4s, v_1,v_2\}$  and
$$F_{\sigma}(t)=(1+t)^{1/\alpha},\quad\sigma=4s,v_1,v_2.$$
For $g\in N$
  with  $g=(4k,n_1,n_2)$,
$$\|(4k,n_1,n_2)\|_{\mathfrak F_N}\asymp \max\{ |k|^{\alpha},|n_1|^{\alpha}, |n_2|^{\alpha}\}.$$
This   confirms   the fact that $\|g\|_{\mathfrak F_G}\asymp \|g\|_{\mathfrak F_N}$ when $g\in N$.
Regarding $\mu$, we have
$$\Lambda_{2,G,\mu}(v)\asymp v^{-\alpha/3},\;\;\mu^{(n)}(e)\asymp n^{-3/\alpha}.$$
\end{exa}

\section{Pseudo-Poincar\'e inequality and control} \setcounter{equation}{0}    \label{sec-control}
We now turn to the extension  of some
 results in   \cite{SCZ-nil0} to random walk driven by measures $\mu\in \mathcal
P_{\preceq}(G,\mbox{reg})$, when $G$ has polynomial volume growth. Some of
these applications are new even in the case of nilpotent groups.

We recall two definitions  from \cite{SCZ-nil0} and a general result involving these definitions.
Note that these statements involve the more restrictive notion of norm rather than the one of quasi-norm.
\begin{defin}
\label{controlled} Let $\mu$ be a symmetric probability measure on a
group $ G $. Let $\|\cdot\|$ be a norm with volume function $V$. Let
$r: (0,\infty)\rightarrow (0,\infty)$ be a continuous   and   increasing function
with inverse $\rho$. Let $(X_n)_0^\infty$ be the random walk on $G$
driven by $\mu$.

\begin{itemize}
\item We say that $\mu$ is $(\|\cdot\|,r)$-controlled if the following
properties are satisfied:
\begin{enumerate}
\item For all $n$, $\mu^{(2n)}(e)\asymp V(r(n))^{-1}.$

\item For all ${\eps}  >0$ there exists $\gamma\in (0,\infty)$ such that
for all $n\geq 1$,
\begin{equation*}
\mathbf{P}_e\left(\sup_{0\le k\le n}\{\|X_k\|\}\ge \gamma r(n)\right)\le
{\eps}  .
\end{equation*}
\end{enumerate}
\item We say that $\mu$ is strongly $(\|\cdot\|,r)$-controlled if the following properties are satisfied:
\begin{enumerate}
\item There exists $C\in (0,\infty)$ and, for any $\kappa>0$, there exists $
c(\kappa)>0$ such that, for all $n\ge 1$ and $g\in G$ with $\|g\|\le
\kappa r(n)$,
\begin{equation*}
c(\kappa)V(r(n)))^{-1} \le \mu^{(2n)}(g)\le C V(r(n))^{-1}.
\end{equation*}

\item There exist ${\eps} , \gamma_1
\in (0,\infty)$ and
$\gamma_2\ge 1$, such that for all $n,\tau$ with $\frac{1}{2}
\rho(\tau/\gamma_1)\le n \le \rho(\tau/\gamma_1) $,
$$\inf_{x:\|x\|\le \tau }
\mathbf{P}_x\left(\sup_{0\le k\le n}
\{\|X_k\|\}\le \gamma_2 \tau; \|X_n\|\le \tau \right)
\ge {\eps}.
$$
\end{enumerate}
\end{itemize}
\end{defin}

\begin{pro}[{\cite[Propositin 1.4]{SCZ-nil0}}]
\label{pro-StC}  Let $\|\cdot\|$ be a norm on $G$. Assume that $r : (0,\infty)\rightarrow (0,\infty)$ is continuous   and  increasing with inverse $\rho$
and that the symmetric probability measure $\mu$ is strongly $(\|\cdot\|, r)$
-controlled. Then, for any $n$ and $\tau$ such that $\gamma_1 r(2n)\ge \tau$, we have
$$\inf_{x:\|x\|\le \tau}
\mathbf{P}_x\left(\sup_{0\le k\le n} \|X_k\| \le \gamma_2 \tau; \,  \|X_n\|\le \tau\right)
\ge {\eps} ^{1+ 2n/\rho(\tau/\gamma_1) }.
$$
\end{pro}

Another key notion is that of pseudo-Poincar\'e inequality (see \cite{CSCiso}).

\begin{defin}
Let $G$ be a discrete group equipped with a symmetric probability
measure $\nu$, a quasi-norm $\|\cdot\|$ and a continuous
increasing
function $r : (0,\infty)\rightarrow (0,\infty)$ with inverse $\rho$. We say that $\nu$ satisfies a pointwise $
(\|\cdot\|,r)$-pseudo-Poincar\'e inequality if, for any $f$ with finite
support on $G$,
$$\forall\, g\in G,\;\;\sum_{x\in G}|f(xg)-f(x)|^2\le C \rho(\|g\|) \mathcal{E}_\nu (f,f).$$
\end{defin}

\begin{pro} \label{pro-PSmu}
Let $G$ be a group of polynomial volume growth and $\mu\in \mathcal
P_{\preceq}(G,\mbox{\em reg})$, see   {\em Definition
\ref{def-Pregorder}}.    Let $(\Sigma_G,\mathfrak F_G)$ be a geometry adapted
to $\mu$ as in {\em Definition \ref{DefxiG}}.  Then the symmetric
probability measure $\mu$ satisfies a pointwise  linear $($i.e.,
$r(t)=t$$)$  $\|\cdot\|_{\mathfrak F_G}$-pseudo-Poincar\'e
inequality.
\end{pro}
\begin{proof}This follows from Proposition \ref{pro-PS3} and Corollary \ref{cor-polvol} via a simple telescopic sum argument and the Cauchy-Schwarz inequality for finite sums.
\end{proof}

\begin{theo}\label{thm0}Let $G$ be a group of polynomial volume growth and $\mu$ be a probability in $ \mathcal P_{\preceq}(G,\rm{reg})$.
  Let $(\Sigma_G,\mathfrak F_G)$ be a geometry adapted to $\mu$ as in {\em
Definition \ref{DefxiG}}.  Let
$$\mathbf F=\mathbf F_{G,\mathfrak F_G}$$
be as in {\em Definition \ref{def-FG}}. Let $\mathfrak F_{G,2}$ be
a system of functions related to the system $\mathfrak F_G$ as in {\em
Proposition \ref{pro-variation}(c)} so that, for each $\sigma\in
\Sigma_G$,
$$F_{2,\sigma} \mbox{ is convex, } \  F_{2,\sigma}\asymp
F_{\sigma}\circ F^{-1} \  \mbox{ and } \  F(t)=(1+t)^{\weight_*}-1$$
 with $\weight_*>0$ as in
{\em Proposition \ref{pro-variation}}{\rm \red (c)}. Then $\|\cdot \|_{\mathfrak
F_{G,2}}$ is a norm on $G$, and it satisfies $\|\cdot \|_{\mathfrak
F_{G,2}}\asymp \|\cdot\|_{\mathfrak F_G}^{\weight_*}$ over $G$.  Set
$$V(R)=  \#\{\|g\|_{\mathfrak F_{G,2}}\le R\}.$$
With this notation, the following properties are satisfied:
\begin{enumerate}
\item[$(1)$] For all $R\ge 1$, $V(R) \asymp \mathbf F(R^{1/\weight_*})$ and,  for all $n$,
$$\mu^{(n)}(e)\asymp 1/\mathbf F(n)\asymp 1/V(
n^{\weight_*}).$$
\item[$(2)$] There exists a constant $C_1$ such that, for all $g\in G$, $f\in L^2(G)$,  $$\sum_{x\in G}|f(xg)-f(x)|^2\le
C_1 \|g\|_{\mathfrak F_{G,2}}^{1/\weight_*}  \mathcal{E}_\mu (f,f).$$
\item[$(3)$]  There exists a constant $C_2$ such that, for all $n,m\in \mathbb N$  and  $x,y\in G$,
$$ |\mu^{(n+m)}(xy)-\mu^{(n)}(x)|\le C_2 \left(\frac{m}{n}+\frac{\|y\|^{1/2\weight_*}_{{\mathfrak F}_{G,2}}}{\sqrt{n}} \right) \mu^{(n)}(e).$$
\item[$(4)$] There exists
$\eta \in (0, 1]$
such that, for all  $n\ge 1$ and $g\in G$ with $\|g\|_{\mathfrak F_{G,2}}\le \eta n^{\weight_*}$, we have
$$\mu^{(n)}(g) \asymp 1/V(n^{\weight_*}).$$
\item[$(5)$]  The symmetric probability measure $\mu$ is $(\|\cdot \|_{\mathfrak F_{G,2}},r)$-controlled with $r(t)= t^{\weight_*}$.
\item[$(6)$]  The symmetric probability measure $\mu$ is strongly $(\|\cdot \|_{\mathfrak F_{G,2}},r)$-controlled with $r(t)= t^{\weight_*}$.\end{enumerate}
\end{theo}
\begin{proof} The first two items are results that have already been proved above and which are now stated in terms of $\mathfrak F_{G,2}$ instead of $\mathfrak F_G$.
The point in doing this is that $ \|\cdot \|_{\mathfrak F_{G,2}}$ is
a norm whereas $\|\cdot\|_{\mathfrak F_G}$ might only be a
quasi-norm.  The third item (regularity) follows from the first two and
  Appendix   \ref{app-reg}, see Proposition \ref{pro-93}.
Item (4) immediately follows from (1) and (3).

Given the relation between $\|\cdot\|_{\mathfrak F_G}$ and $\|\cdot\|_{\mathfrak F_{G,2}}$, item (5) follows from (1), the estimate on $\Lambda_{2,G,\mu}$ stated in Theorem \ref{th-polmu} and \cite[Lemma 4.1]{PZ}.

To prove item (6), observe that the norm $\|\cdot\|_{\mathfrak F_{G,2}}$
is well-connected in the sense of \cite[Definition 3.3]{SCZ-nil0}
(this is a rather weak property satisfied by any such norm), and that the pointwise pseudo-Poincar\'e inequality  stated as item (2) holds.  Because of these properties and \cite[Proposition 3.5]{SCZ-nil0},   $(\|\cdot
\|_{\mathfrak F_{G,2}},r)$-control (that is, item (5)) implies
 strong $(\|\cdot \|_{\mathfrak F_{G,2}},r)$-control (item (6)).
\end{proof}

\medskip

Note that, by the symmetry of $\mu$, $\|\mu^{(2n)}\|_\infty=\mu^{(2n)}(e)$. Since $n\mapsto \|\mu^{(n)}\|_\infty$ is obviously non-increasing, the on-diagonal estimate
in Theorem \ref{thm0}(1) implies
\begin{equation}\label{e:5.4}
\|\mu^{(n)}\|_\infty\asymp 1/\mathbf F(n).
\end{equation}

\section{H\"{o}lder continuity of
caloric
 functions} \label{sec-Holder}  \setcounter{equation}{0}
In this section, we place ourselves
in the context of Theorem \ref{thm0}.
Namely, we consider a group $G$, finitely generated and of polynomial volume growth, equipped with a probability measure $\mu\in \mathcal P_{\preceq}(G,\mbox{reg})$. Having chosen a normal nilpotent subgroup $N$ of finite index in $G$,  we construct the quasi-norm \break
 $\|\cdot\|_{\mathfrak F_G}$   as in Definitions \ref{DefxiG}  and the associated function $\mathbf F=\mathbf F_{G,\mathfrak F_G}$ introduced in  Definition \ref{def-FG} so that
 $\#\{g\in : \|g\|_{\mathfrak F_G}\le R\} \asymp \mathbf F(R)$, see Corollary \ref{cor-polvol}.   Consider further the norm $\|\cdot\|_{\mathfrak F_{G,2}}$ and positive real $\weight_*$ as in Theorem \ref{thm0} so that $\|\cdot \|_{\mathfrak F_{G,2}}\asymp \|\cdot\|_{\mathfrak F_G}^{\weight_*}$.  For any $x\in G$ and $r>0$, let $$B(x, r)=\{z\in
G: \|x^{-1}z\|_{\mathfrak F_{G,2}}<r\}.$$
Also, for any $x\in G$ and $A\subset G$, let
$$\|x^{-1}A\|_{\mathfrak F_{G,2}}=\inf_{y\in A}\|x^{-1}y\|_{\mathfrak F_{G,2}}$$
  denote  the
distance between $x$ and $A$ with respect to $\|\cdot\|_{\mathfrak
F_{G,2}}.$

 The goal of this section is to prove the H\"older regularity of bounded solutions of the discrete parabolic equation  (\ref{diffeq}). More precisely, given
 a (discrete) time interval  $I=[0,N]$  and a subset $A$ of $G$, we say that a real valued bounded function $q$ defined on $I\times G$ is a solution of (\ref{diffeq}) in $I\times A$ if
 $$q(n+1,x)-q(n,x) =
 [q(n, \cdot)*(\mu - \delta_e )](x), \quad n,n+1\in I, \, x\in A,
 $$
 or, equivalently
$$ q(n+1, x)= q(n,\cdot)*\mu (x),\quad \;n,n+1\in I,x\in A.$$
 for all $n$ with $n,n+1\in I$ and all $x\in A$.  Note that this definition requires $q$ to be defined over all of $G$ at all times in $I$. This is natural in the present context since, typically, the measure $\mu$ has infinite support. Because such a solution $q$ is bounded and defined for each $k\in I$ on all of $G$, it is always possible to extend it forward in time.

  Let $(X_n)_0^\infty$ denote the random walk driven by $\mu$  started at $X_0$. For any subset
$A\subset G$, define
$$
\tau_A=\inf\{n\ge0: X_n\notin A\},
\quad \sigma_A=\inf\{n\ge1: X_n\in A\}.
$$

 The argument developed below is mainly based on that of \cite[Theorem 4.9]{BK},
 which in turn is the discrete analog of that for \cite[Theorem 4.14]{CK},
  and makes use of the following notion of
 {\em caloric function}
 (we followed the definition from  \cite{CK, BK}, though  co-caloric might be
   more appropriate, see \cite[page 263]{Doob}).

 Let $\Upsilon=\Z_+ \times G$. We will make use of the $\Upsilon$-valued
Markov chain $Z_k:=(V_k, X_k)$ where $V_k=V_0+k$. Write
$\Pp_{(i,x)}$ for the law of $Z_k$ started at $(i,x)$ and let
$\mathscr{F}_i=\sigma\{Z_k:k\le i\}$. A bounded function $u(k,x)$ on
$\Upsilon$ is said to be
{\em caloric}
 on $D\subset \Upsilon$, if
$k\mapsto u(V_{k\wedge \tau_D}, X_{k\wedge \tau_D})$
 is a martingale, where
$$
\tau_D=\inf \left\{k\ge 0: (V_k,X_k)\notin D \right\}.
$$
In the case $V_0=0$ and $D=I\times A=[k_1,k_2]\times A$, $k_1<k_2$, the condition that $k\mapsto u(V_{k\wedge \tau_D}, X_{k\wedge \tau_D})$ is a martingale is equivalent to
$$ u(k-1,x)=[u(k,\cdot)*\mu] (x),\;\; k,k-1\in I, \;x\in A.$$
This
is the ``backward version'' of the parabolic discrete equation (\ref{diffeq}) in the sense that, for any $N$,  $q(k,x)=u(N-k,x)$ is a solution of (\ref{diffeq})  in   $J\times A$, $J=[N -k_2,N-k_1]$ (and vice versa). In particular, for any fixed integer $N$ and bounded function $f$ on $G$, the functions
$$(k,x)\mapsto  \mu^{(N-k)}(x) \mbox{ and }  (k,x)\mapsto  f*\mu^{(N-k)}(x)$$
are
caloric
 functions on $[0,N]\times G$.

 \begin{lem}\label{Lem-1} Given any $\delta>0$, there exists a
constant $\kappa:=\kappa(\delta)>0$ such that for
any $x,y \in G$,  $A\subset G$  with  $\min \left\{ \|x^{-1}A\|_{\mathfrak F_{G,2}},
 \|y^{-1}A\|_{\mathfrak F_{G,2}} \right\} \geq\kappa n^{\weight_*}$, and any $n\geq 1$,
$$
\Pp_x(X_n=y, \sigma_A\le n)\le \delta /\mathbf F(n).
$$
\end{lem}

\begin{proof}
By the strong Markov property of $X_n$ and \eqref{e:5.4},  we have
\begin{align*}
\Pp_x(X_n=y, \sigma_A\le [n/2])
&= \Ee_x \left[ {\bf1}_{\{\sigma_A\le [n/2]\}} \Pp_{X_{\sigma_A}}(X_{n-\sigma_A}=y) \right] \\
&\le \frac{c_1}{\mathbf F(n-[n/2])} \Pp_x(\sigma_A\le [n/2])\\
&\le \frac{c_2}{\mathbf F(n)}\Pp_x(\sigma_A\le [n/2]).
\end{align*}
Furthermore, due to Theorem \ref{thm0}(5), by choosing $\kappa:=\kappa(\delta)>0$ large
enough, we have
$$\Pp_x(\sigma_A\le [n/2])\le \Pp_x\left(\sup_{0\le k\le [n/2]}\|x^{-1}X_k\|_{\mathfrak F_{G,2}}\ge \kappa n^{\weight_*}\right)\le \delta/(2c_2),$$ which yields that
\begin{equation}\label{p-lem-1}\Pp_x(X_n=y, \sigma_A\le [n/2])\le \delta /(2\mathbf F(n)).\end{equation}

We now consider $\Pp_x(X_n=y, [n/2]\le \sigma_A\le n)$. If the first
hitting time of
$A$ is
between time $[n/2]$ and time $n$, then
the last hitting time of $A$ before time
$n$ is larger than
$[n/2]$. So,
setting $\hat\sigma_A=\sup\{k\le n: X_k\in A\}$, we have
$$
\Pp_x \left( X_n=y, [n/2]\le \sigma_A\le n \right) \le \Pp_x \left( X_n=y, [n/2]\le \hat\sigma_A\le n \right).
$$
 We claim that by  time reversal,
$$\Pp_x(X_n=y, [n/2]\le \hat\sigma_A\le n)=\Pp_y(X_n=x, \sigma_A\le n-[n/2]).$$ To see this, observe that the symmetry of the transition probability $p(x,y):=\mu(x^{-1}y)$ implies that we have, for any $[n/2]\le k\le n$,
\begin{align*}&\Pp_x(X_k=z_k,X_{k+1}=z_{k+1},\cdots, X_n=y)\\
&=p_k(x,z_k)p(z_k,z_{k+1})\cdots p(z_{n-1},y)\\
&=\Pp_y(X_1=z_{n-1},
\cdots, X_{n-k}=z_k, X_n=x),
\end{align*} where
$p_k(x,z):=\Pp_x(X_k=z)=\mu^{(k)}(x^{-1}y)$. Summing over all $z_k\in A$ and
$z_{k+1},\cdots, z_{n-1}\notin A$, we have
\begin{align*}&\Pp_x(X_k\in A, X_{k+1}\notin A,\cdots, X_{n-1}\notin A, X_n=y)\\
&= \Pp_y(X_1\notin A, \cdots,
X_{n-k-1}\notin A, X_{n-k}\in A,
X_n=x).\end{align*}
Further, summing over $[n/2]\le k\le n$, this
yields that
$$\Pp_x([n/2]\le \hat\sigma_A\le n, X_n=y)= \Pp_y(0\le \sigma_A\le n-[n/2], X_n=x).$$ This proves the desired assertion.
Arguing as in the first part of the proof, we find that
$$\Pp_x([n/2]\le \hat\sigma_A\le n, X_n=y)=\Pp_y(0\le \sigma_A\le
n-[n/2], X_n=x)\le \delta /(2\mathbf F(n)).$$ Therefore, the lemma
follows from the estimate above and \eqref{p-lem-1}.\end{proof}

\begin{pro} Let
$\eta\in(0,1]$
 be the constant in
Theorem $\ref{thm0}(4)$.
 For all $n\ge 1$, there exist constants $c_1\in (0,\infty)$ and $\theta\in (0,1)$ such that
for any $x,y,z\in G$ with $\max \left\{\|x^{-1}z\|_{\mathfrak F_{G,2}},  \|y^{-1}z\|_{\mathfrak F_{G,2}} \right\} \le \eta n^{\weight_*}/2$
and $r\ge
(n/\theta)^{\weight_*}$,
$$
\Pp_x \left(X_n=y, \sup_{0\le k\le n}\|z^{-1}X_k\|_{\mathfrak F_{G,2}}\le r \right)\ge \frac{c_1}{\mathbf F(n)}.
$$
\end{pro}
\begin{proof}
Note that since $\| x^{-1}y\|_{\mathfrak F_{G,2}} \leq \| x^{-1}z\|_{\mathfrak F_{G,2}} + \| y^{-1}z\|_{\mathfrak F_{G,2}} \leq \eta n^{\weight_*}$,
we have by  Theorem \ref{thm0}(4) that
$$
\Pp_x(X_n=y)\ge \frac{c_0}{\mathbf F(n)}.
$$
 Let $\delta=c_0/2$ in Lemma \ref{Lem-1}. Then for all
$r>(\kappa+1)n^{\weight_*}$ (where $\kappa$ is the constant in Lemma
\ref{Lem-1} associated with this $\delta$), we have
by Lemma \ref{Lem-1} that
\begin{align*}&\Pp_x(X_n=y,\sup_{0\le k\le n}\|z^{-1}X_k\|_{\mathfrak F_{G,2}}\le r)\\
&=\Pp_x(X_n=y)-\Pp_x(X_n=y, \sup_{0\le k\le
n}\|z^{-1}X_k\|_{\mathfrak F_{G,2}}> r)  \\
&\geq \Pp_x(X_n=y)-\Pp_x(X_n=y, \sigma_{B(z,r)^c}\le n)  \\
&\ge c_0 /(2\mathbf F(n)),
\end{align*} where in the last inequality we used the facts that $$\|x^{-1}B(z,r)^c\|_{\mathfrak F_{G,2}}\ge r-\|x^{-1}z\|_{\mathfrak F_{G,2}}\ge (\kappa+\eta)n^{\weight_*}- \eta n^{\weight_*}/2\ge \kappa n^{\weight_*}$$ and, similarly, $\|y^{-1}B(z,r)^c\|_{\mathfrak F_{G,2}}\ge \kappa n^{\weight_*}.$
This proves the desired assertion
with $\theta=(\kappa+1)^{-1/\weight_*}.$ \end{proof}

The following is
an immediate consequence of the last proposition.
\begin{cor}\label{Cor-1} Let $\eta \in (0, 1]$ be the constant in
Theorem  $\ref{thm0}(4)$.
There exist constants $\theta\in (0,1)$ and $c_1>0$ such that
for  every $z\in G$,   $n\geq 1$,
  and $A\subset B(z, \eta n^{\weight_*}/2)$,
$$
\Pp_x \left(X_n\in A,  \,  \tau_{B(z, (n/\theta)^{\weight_*})}>n \right)
\ge c_1 \frac{ \# A}{\mathbf F(n)} \quad \hbox{ for every }
x\in B(z, \eta n^{\weight_*}/2).
$$
 \end{cor}

The following is a key proposition concerning
the space-time process $Z_k=(V_k,X_k) = (V_0+k, X_k)$
on $\Upsilon=\Z_+ \times G$
discussed before Lemma \ref{Lem-1},
which will be used in the proof  of Theorem \ref{th-Holder}.

\begin{pro}\label{Ke}
Let $\eta \in (0, 1]$ be the constant in Theorem $\ref{thm0}(4)$,
and
$\theta\in (0,1)$ be the constant in Corollary $\ref{Cor-1}$. Let
$m$ be the counting measure on $\Upsilon$, and set
\begin{equation}\label{e:6.2}
C_0 =
2^{3+4\weight_*} / (\eta (\theta \wedge \theta^{\weight_*}))
\quad \hbox{ and } \quad C_1 = \eta \theta^{\weight_*}/
2^{1+4\weight_*}.
\end{equation}
  For every $\delta\in (0,1)$
 and $\gamma \geq C_0$, there is a constant
$c_0=c_0(\delta, C_0)>0$  (independent of $\gamma$)
such that for any $x_0\in G$, $n_0\geq 0$,
$R>1$ with $\theta R^{1/\weight_*}\ge1$  and
$$
A\subset \left[ n_0+  [ \tfrac12 \theta  (\gamma R)^{1/\weight_*}]  , n_0+ [ \tfrac12 \theta  (\gamma R)^{1/\weight_*}] + [ \theta  R^{1/\weight_*}] \right]
\times B(x_0,   C_1\gamma R)
$$
 satisfying
$$
\frac{ m (A)}{   [ \theta  R^{1/\weight_*}]  \cdot \# B(x_0, C_1\gamma  R) }\ge \delta,
$$
we have
$$
\Pp_{z}  (\sigma_A<\tau_{Q(n_0,x_0, \gamma R)})\ge c_0 \quad \hbox{for every } z \in  \left[ n_0, n_0+[\theta R^{1/\weight_*}] \right] \times B(x_0, R).
$$
 Here  $Q(n_0,x_0, r):=\left[ n_0,n_0+ [\theta r^{1/\weight_*}] \right]\times B(x_0,   r)$, and
 by abusing the notation, $\sigma_A := \inf\{k\geq 1: Z_k\in A\}$.
\end{pro}

\begin{proof} First note that
by the definition \eqref{e:6.2} of $C_0$,
$C_0 \ge 8^{\weight_*+1} / (\eta (\theta \wedge \theta^{\weight_*})$,
 and so  for any $\gamma \geq C_0$,
$$
 [  \theta (\gamma R)^{1/\weight_*}] > 8 [\theta R^{1/\weight_*}]  .
 $$
  Let $A$ be the subset of $\Upsilon=\Z_+ \times G$ in the proposition.
 For $j\in \mathbb{N}$, let $A_j=\{x \in G: (j, x)\in A\}$.
 There exists some
 $$
 k\in \left[ n_0+  [ \tfrac12 \theta  (\gamma R)^{1/\weight_*}]  , n_0+ [ \tfrac12 \theta  (\gamma R)^{1/\weight_*}] + [ \theta  R^{1/\weight_*}] \right]
 $$
  so  that
\begin{equation}\label{eq:deib2}\begin{split}
 \#A_k
 &\geq  \frac{m(A)}{ [ \theta   R^{1/\weight_*}]} \geq   \delta  \cdot \# B(x_0,  C_1\gamma  R).
\end{split}
\end{equation}
 Note that for any $z=(n, x)\in  \left[ n_0, n_0+[\theta R^{1/\weight_*}] \right] \times B(x_0, R)$,
 \begin{equation}\label{e:6.5}
 \frac58 [ \theta (\gamma R)^{1/\weight_*}] \geq k-n \geq
 \frac18 [ \theta (\gamma R)^{1/\weight_*}].
\end{equation}
In particular,
\begin{equation}\label{e:6.6}
\eta (k-n)^{\weight_*}/2 \geq C_1 \gamma R \geq 4R \quad \hbox{  and } \quad ((k-n)/\theta)^{\weight_*} <\gamma  R.
\end{equation}
It follows from Theorem \ref{thm0}(1) and \eqref{e:6.5} that  $ \# B(x_0, \gamma R) \asymp {\mathbf F(k-n)}$.
Hence we  have by  \eqref{eq:deib2}, \eqref{e:6.6}  and Corollary \ref{Cor-1} that
$$
\Pp_{(n  ,x)} \left(\sigma_A<\tau_{Q(n_0 ,x_0, \gamma R)} \right)
\ge\Pp_{x}(X_{k-n }\in A_k,
\tau_{B(x_0, \gamma R)}>k-n )
\ge c_0.
$$
This completes the proof of the proposition.
\end{proof}

 \medskip

The following is a special case of the L\'evy system formula for Markov chains.
For any $(k,x)\in \Upsilon$ and $A\subset \Upsilon$,
define $N_A(k,x)=\Pp_{(k,x)}(X_1\in A(k+1))$ if $(k,x)\notin A$ and
$0$ otherwise.

\begin{lem}\label{levy} For the $\Upsilon$-valued Markov chain $(V_k,X_k)$, let $A\subset \Upsilon$ and
$$J_n={\bf1}_A(V_n,X_n)-{\bf1}_A(V_0,X_0)-\sum_{k=0}^{n-1} N_A(V_k,X_k).
$$
Then $\{J_{n\wedge \sigma_A}; n\in \mathbb{N}\}$ is a martingale.
\end{lem}

\begin{proof}See \cite[Lemma 3.2]{BL} for the proof. \end{proof}
We also need the following lemma.

\begin{lem}\label{uppexit} For $r>0$, there exists a constant $c_1 >0$ such that
$$\Ee_x(\tau_{B(x,r)})\le c_1r^{1/\weight_*}.$$
\end{lem}

\begin{proof}
By \eqref{e:5.4}
and Theorem \ref{thm0}(1),
for any $x,y\in G$ and $n\ge 1$, we have
$$\Pp_y(\tau_{B(x,r)}>n)\le \Pp_y(X_n \in B(x,r))\le \frac{c_1}{\mathbf F(n)}  \mathbf F(r^{1/\weight_*}).$$ By taking $n=c_2 r^{1/{\weight_*}}$ for a proper constant $c_2>0$,
 $\Pp_y(\tau_{B(x,r)}>n)\le 1/2.$ Using the Markov property at time $kn$ for $k=1,2,\cdots$,
\begin{align*}
\Pp_x(\tau_{B(x,r)}>(k+1)n)&\le \Ee_x\left(\Pp_{X_{kn}}(\tau_{B(x,r)}>n); \tau_{B(x,r)}>kn\right)\\
&\le  \frac{1}{2}\Pp_x(\tau_{B(x,r)}>kn).
\end{align*}
By induction,
$$\Pp_x(\tau_{B(x,r)}>kn)\le 2^{-k}$$ for $k=1,2,\cdots$. With this choice of $n$, we obtain that
$$\Ee_x\tau_{B(x,r)}\le\sum_{k=1}^\infty kn\Pp_x((k-1)n<\tau_{B(x,r)}\le kn)\le n\sum_{k=1}^\infty k 2^{k-1}=:cn,$$ which proves our result.
\end{proof}
\begin{pro}\label{Ee} There is a constant $c_0>0$ such that for any $s\ge 2r$ and $x\in G$,
$$\Pp_{x}(X_{\tau_{B(x,r)}}\notin B(x,s))\le c_0\left(\frac{r}{s}\right)^{1/\weight_*}.$$ \end{pro}

\begin{proof}  Applying Lemma \ref{levy}
with $A=\Z_+\times B(x,s)^c$ at stopping times $n\wedge \tau_{B(x,r)}$ and then letting $n\to \infty$,
we see that
\begin{align*}\Pp_{x}(X_{\tau_{B(x,r)}}\notin B(x,s))&=\Ee_{x}\left(\sum_{k=0}^{\tau_{B(x,r)}-1} N_{B(x,s)^c}(X_k)\right)\\\
&=\Ee_{x}\left(\sum_{k=0}^{\tau_{B(x,r)}-1} \sum_{y\in {B(x,s)^c}}\mu(X_k^{-1}y)\right) \\
&\le \Ee_x\tau_{B(x,r)} \sup_{z\in B(x,r)}\sum_{y\in  B(x,s)^c} \mu(z^{-1}y)\\
&\le c_1 r^{1/\weight_*} s^{-1/\weight_*},\end{align*} where in the last
inequality we used Lemma \ref{uppexit} and the fact that
\begin{equation}\label{e:decay}\sum_{\|h\|_{\mathfrak F_{G,2}}\ge r}\mu(h)\le
c_2r^{-1/\weight_*},\quad r>0.\end{equation}
Note that, \eqref{e:decay}
can be obtained following the same line of reasoning  as for  \eqref{eq:feobaqz}, by using
Theorem \ref{thm0}(1), $\|\cdot \|_{\mathfrak
F_{G,2}}\asymp \|\cdot\|_{\mathfrak F_G}^{\weight_*}$, $F_{N,s}=\Phi_i^{-1}$ (see Definition \ref{DefxiN})
and Remark \ref{rem-Phi}.
 \end{proof}

From now, we take $\theta\in (0, 1) $ and
$C_0,C_1$ be the constants in Corollary \ref{Cor-1}
and  \eqref{e:6.2}, respectively. For $n\in \mathbb{N}$, $x\in G$ and $r>1$ with $\theta r^{1/\weight_*}\geq 1$,
define
$$
Q(n,x,r)=[n,n+ [\theta  r^{1/\weight_*}] ]\times B(x,   r) .
$$
The proof of the following theorem is similar to the proofs  of  \cite[Theorem
4.14]{CK} and \cite[Theorem 4.9]{BK}.
We remark that we also took this opportunity to correct an error in   selecting   subsets
$A$ and $A'$ in the proofs of   \cite[Theorem 4.14]{CK} and \cite[Theorem 4.9]{BK}.
Such a correction was previously made in the proof of \cite[Theorem 6.3]{CL18} on H\"older regularity of
caloric functions for certain diffusion processes.
Our proof  is based on
Propositions \ref{Ke} and \ref{Ee}.

\begin{theo} \label{th-Holder}
Let
$C_0>0$ be the constant in \eqref{e:6.2}.
There are constants $C>0$ and $\beta>0$ such that for any
$R>1$
 with $\theta R^{1/\weight_*}\ge 1$ $($where $\theta$ is
the constant in Corollary $\ref{Cor-1}$$)$ and bounded
caloric
function $q$ in
$Q(0,x_0, C_0R)= [0, [\theta  (C_0R)^{1/\weight_*}]  ]  \times B(x_0, C_0     R)$, we have
$$|q(m_1,x)-q(m_2,y)|\le C\|q\|_\infty \left(\frac{|m_1-m_2|^{\weight_*} +\|x^{-1}y\|_{\mathfrak F_{G,2}}}{R}\right)^\beta$$
for all $(m_1,x)$, $(m_2, y)\in Q(0,x_0,R)$,
 and
 $$
 \|q\|_\infty=\sup_{(i,x)\in [0,[\theta (C_0 R)^{1/\weight_*}]]\times G}  q(i,x).
 $$
In particular, we have
$$\sum_{x\in G}|\mu^{(m_1)}(x)-\mu^{(m_2)}(xy)|\le C
\left(\frac{|m_1-m_2|^{\weight_*}+\|y\|_{\mathfrak F_{G,2}}}{  n_0 ^{\weight_*}}\right)^{\beta}$$
for all $y\in G$ and $m_1,m_2\ge  n_0   \ge1$.
\end{theo}

\begin{proof} Recall that $Z_k=(V_k,X_k)$ is the space-time process of $X$, where $V_k=V_0+k$. Without loss of generality, assume that $0\le q(z)\le \|q\|_\infty=1$ for all
$z\in [0, [\theta (C_0R)^{1/\weight_*}] ]\times G$.

By Proposition \ref{Ke}, there exists a  constant $c_1\in (0, 1)$ such that if $x_0\in G$, $n_0\ge0$,
 $r>1$ with $\theta r^{1/\weight_*}\ge 1$, $\gamma \geq C_0$, and
 $$
A\subset \left[ n_0+ [ \tfrac12 \theta (\gamma r)^{1/\weight_*}] , n_0+ [ \tfrac12 \theta  (\gamma r)^{1/\weight_*}] + [ \theta  r^{1/\weight_*}] \right]
\times B(x_0, C_1 \gamma r)
$$
  satisfying
$$
\frac{ m (A)}{[ \theta  r^{1/\weight_*}] \cdot \# B(x_0,C_1\gamma r)}\ge 1/3,
$$
then
\begin{equation}\label{e1}
\Pp_{z}  (\sigma_A<\tau_{Q(n_0,x_0, \gamma r)})\ge c_1
\quad \hbox{for } z \in  \left[ n_0, n_0+[\theta r^{1/\weight_*}] \right] \times B(x_0, r).
\end{equation}
Here $C_1=
\eta \theta^{\weight_*}/2^{1+4\weight_*}\in (0, 1)$ as defined in \eqref{e:6.2}, where $\eta \in (0, 1]$ is the constant in Theorem $\ref{thm0}(4)$.
 On the other hand, according to Proposition \ref{Ee}, there is a
constant $c_2>0$ such that if any $x\in G$,   and $s\ge 2r$, then
\begin{equation}\label{e2}\Pp_{x}(X_{\tau_{B(x,r)}}\notin B(x,s))\le c_2 (r/s)^{1/\weight_*}.
\end{equation}

Let $\eta_0=1-  (c_1/4)  $ and
$$
 \rho = C_0^{-1}
 \wedge    (\eta_0/2)^{\weight_*}\wedge (c_1\eta_0/(8c_2))^{\weight_*}<1.
$$
Note that
for every $(n_0 ,x)\in Q(0,x_0,R)$, $q$ is
caloric in $Q(n_0,x , R)\subset Q(0,x_0,C_0R)$.
We will show that
\begin{equation}\label{e3}
\sup_{Q(n_0, x , \rho^{2k}
R)}q-\inf_{Q(n_0,x ,\rho^{2k} R)}q\le \eta_0^k
\end{equation} for all
$k\leq K_0$,  where $K_0$ is the largest integer  $k$ so that $\theta  (\rho^{2k}R)^{1/\weight_*}
\ge 1.$
  For notational convenience, we write $Q_i$ for
$Q(n_0,x ,\rho^iR)$ and $\tau_i=\tau_{Q(n_0 , x ,\rho^iR)}$. Define
$$
a_i=\sup_{Q_{2i}}q,\qquad b_i=\inf_{Q_{2i}}q.
$$
Clearly $b_i-a_i\le 1\le \eta_0^i$ for all $i\le 0$.
Now suppose that $b_i-a_i\le \eta_0^i$ for all $i\le k$ and we are going to show that
 $b_{k+1}-a_{k+1}\le \eta_0^{k+1}$ as long as $k+1\leq K_0 $.
Observe that $Q_{2k+2} \subset Q_{2k+1}  \subset Q_{2k}$ and $a_k\le q \le b_k$ on $Q_{2k }$. Define
 \begin{align*}
   Q_{2k+1 }' = &\left[ n_0+ [ \tfrac12 \theta (\rho^{2k +1}R)^{1/\weight_*}],  n_0+ [ \tfrac12\theta  (\rho^{2k+1}R)^{1/\weight_*}] +
   [ \theta (\rho^{2k +2}R)^{1/\weight_*}]
   \right]\\
   & \times B(x_0, C_1 \rho^{2k+1} R)
\end{align*}
and
$$
A'= \left\{ z \in   Q_{2k+1 }' : q(z) \leq (a_k+b_k)/2 \right\}.
$$
It is clear that $Q_{2k+1}'\subset Q_{2k+1}$.
 We may suppose that
 $$
 \frac{m(A')}{ [ \theta (\rho^{2k +2}R)^{1/\weight_*}] \cdot \# B(x_0, C_1 \rho^{2k+1} R)  }\ge 1/2.
 $$
Otherwise, we use $1-q$ instead of $q$. Let $A$ be a compact subset of $A'$ such that
 $$
 \frac{m(A)}{ [ \theta (\rho^{2k +2}R)^{1/\weight_*}] \cdot \# B(x_0, C_1 \rho^{2k+1} R)  } \geq 1/3.
 $$
For any given $\eta_0>0$, pick $z_1,z_2\in Q_{2(k+1)}$ so that $q(z_1)\ge b_{k+1}-\varepsilon$ and $q(z_2)\le a_{k+1}+\varepsilon$. Then, according to \eqref{e1}, \eqref{e2} and \eqref{e3},
\begin{align*}
& b_{k+1}-a_{k+1}-2\varepsilon \\
&\le q(z_1)-q(z_2)\\
&=\Ee_{z_1}[q(Z_{\sigma_A\wedge\tau_{2k+1}})-q(z_2)]\\
&=\Ee_{z_1}[q(Z_{\sigma_A})-q(z_2); \sigma_A<\tau_{2k+1 }]  \\
&\quad + \Ee_{z_1}[q(Z_{\tau_{2k+1}})-q(z_2); \sigma_A> \tau_{2k+1}, Z_{\tau_{2k+1}}\in Q_{2k}]\\
&\quad+ \sum_{i=1}^\infty \Ee_{z_1}[q(Z_{\tau_{2k+1}})-q(z_2); \sigma_A\ge \tau_{2k+1}, Z_{\tau_{2k +1}}\in Q_{2(k-i)}\setminus Q_{2(k+1-i)}]\\
&\le\left(\frac{a_k+b_k}{2}-a_k\right)\Pp_{z_1} (\sigma_A< \tau_{2 k+1 })+(b_k-a_k) \Pp_{z_1} (\sigma_A\ge \tau_{2k+1})\\
&\quad+\sum_{i=1}^\infty (b_{k-i}-a_{k-i}) \Pp_{z_1}( Z_{\tau_{2k+1}}\notin Q_{2(k+1-i)})\\
&\le (b_k-a_k)\left(1-\frac{\Pp_{z_1} (\sigma_A< \tau_{2k+1}) }{2}\right)+\sum_{i=1}^\infty c_2 \eta_0^k (\rho^{1/\weight_*}/\eta_0)^i\\
&\le(1-c_1/2) \eta_0^k+ 2c_2\eta_0^{k-1} \rho^{1/\weight_*}\\
&\le (1-c_1/2) \eta_0^k+c_1\eta_0^k/4= \eta_0^{k+1}.
\end{align*}
Since $\varepsilon$ is arbitrary, we have $b_{k+1}-a_{k+1}\le
\eta_0^{k+1}$, and this proves the claim \eqref{e3}.

\medskip

For $z=(i,x)$ and $w=(j,y)$ in $Q(0,x_0,R)$ with $i\le j$, let $k$
be the largest integer such that
$$\|z-w\|:=  ( |j-i|/\theta )^{\weight_*}+ \|x^{-1}y\|_{\mathfrak F_{G,2}}\le
\rho^{2k}R.
$$
 Then  $\log (\|z-w\|/R)\ge 2(k+1)\log \rho$, $w\in Q(n,x,
\rho^{2k} R)$ and
$$|q(z)-q(w)|\le \eta_0^k=e^{k\log \eta_0}\le c_3\left(\frac{\|z-w\|}{R}\right)^{\log \eta_0/ (2\log \rho)}.
$$
This proves the first desired assertion.

\medskip

Fix $n_0\ge 1$, $N_0\ge2$ and a bounded function $u$ on $G$ with $\|u\|_\infty=1$.  Set
$q(n,x)=\sum_z  u(z)p_{N_0-n}(z,x)=u*\mu^{(N_0-n)}(x)$. This function $q$ is  a
caloric
function on $[0,N_0-n_0]$ (for example see \cite[Lemma 4.5]{CK}),
and is bounded above by $\|u\|_\infty=1$.
Take $R>1$ such that  $ \theta^{\weight_*} R=n_0^{\weight_*}$.
 (Note that in particular $\theta R^{1/\weight_*}\ge 1$ so the first assertion will apply.)
 Let
$m_1,m_2\in [n_0,N_0]$ with $m_1>m_2$ and $x_1,x_2\in G$. Assume
first that
\begin{equation}\label{e4}|m_1-m_2|^{\weight_*}+\|x_1^{-1}x_2\|_{\mathfrak F_{G,2}}<
\theta^{\weight_*} R
=n_0^{\weight_*}
\end{equation}
and so $(N_0-m_2,x_2)\in Q(N_0-m_1,x_1,R)\subset [0,N_0-n_0]\times
G$. Applying the first assertion to this
caloric
function $q(n,x)$
with $(N_0-m_1,x_1)$, $(N_0-m_2,x_2)$ and $Q(N_0-m_1,x_1,R)$ in
place of $(m_1,x)$, $(m_2,y)$ and $Q(0,x_0,R)$ respectively, we have
$$|u*\mu^{(m_1)}(x^{-1}z)-u*\mu^{(m_2)}(y^{-1}z)|\le \frac{c}{n_0^{\beta \weight_*}}\left( |m_1-m_2|^{\weight_*}+\|x_1^{-1}x_2\|_{\mathfrak F_{G,2}}\right)^\beta.$$
This inequality is also trivially
true when \eqref{e4} does not hold. So the inequality above
holds for every $m_1,m_2\in [n_0,N_0]$ and $x_1,x_2\in G$ for all
$n_0\ge1$ and $N_0\ge 2$. This proves the second assertion after taking the supremum over all $u$ with $\|u\|_\infty=1$.
\end{proof}

\begin{rem} From \eqref{th-Holder} we can get that
$$|\mu^{(m_1)}(x)-\mu^{(m_2)}(y)|\le \frac{c}{n_0^{\beta \weight_*}\mathbf F( n_0) }
\left(|m_1-m_2|^{\weight_*}+\|x^{-1}y\|_{\mathfrak F_{G,2}}\right)^\beta$$
for any $x,y\in G$ and $m_1,m_2\ge n_0\ge1$. However, this assertion is weaker than that in Theorem \ref{thm0}(3).
\end{rem}

As an easy application of
Theorem \ref{th-Holder},
we have the following.
Recall that $G$ is a  finitely generated group of polynomial volume growth, equipped with a probability measure $\mu\in \mathcal P_{\preceq}(G,\mbox{reg})$.

\begin{cor} The pair
$(G,\mu)$ has the   $($weak$)$   Liouville property, namely,
all bounded $\mu$-harmonic functions on $G$ are constant.  Moreover there are constants $C,\beta>0$ such that, for any $x_0\in G,R>1$ and any bounded function
$u$ defined on $G$ and $\mu$-harmonic in $B(x_0,R)=\{z\in G: \|x_0^{-1}z\|_{\mathfrak F_{G,2}}\le R\}$, we have
\begin{equation}\label{harmest}
\forall\,x,y\in B(x_0,R/2),\;\; |u(x)-u(y)| \le C\left(\frac{ \|x^{-1}y\|_{\mathfrak F_{G,2}}}{R}\right)^\beta \|u\|_\infty.
\end{equation}
\end{cor}

\begin{rem} For finitely generated nilpotent groups and any probability measure $\mu$, the weak Liouville property was proved in  \cite{DM}. In the case of symmetric probability measures on finitely generated nilpotent groups, the strong Liouville property (i.e., all non-negative $\mu$-harmonic functions are constant) follows from \cite{Marg}.
Because any finitely generated group of polynomial growth contains a nilpotent subgroup of finite index, these two results extend to groups of
polynomial growth  (given a symmetric measure on a group of polynomial volume growth $G$, one constructs a symmetric measure on a nilpotent
subgroup of finite index $N$ so that the restriction to $N$ of any harmonic function on $G$ is harmonic on $N$). In addition, for finitely generated group of polynomial volume growth and measures $\mu\in \mathcal P_{\preceq}(G,\mbox{reg})$
as treated here, the weak Liouville property follows also from a more general and direct argument given in \cite[Corollary 2.3]{EZ}.  The different methods
used in these papers do not provide estimates such as  (\ref{harmest}).
\end{rem}

\appendix

\section{Space-time regularity for $\mu^{(n)}$} \label{app-reg}
This section provides details concerning the intrinsic regularity afforded to the iterated convolutions of a symmetric measure, providing a straightforward extension and complement to \cite[Theorem 4.2]{Hebisch1993}.

\begin{lem} Fix ${\eps}  \in (0,1]$ and $\alpha>0$. If $\mu$ is symmetric and the spectrum of $f\mapsto f*\mu$ on $L^2(G)$ is contained in $[-1+{\eps} ,1]$,
  then   there exists a constant $C_{{\eps} ,\alpha}$ such that
$$ \forall\,f\in L^2(G),\;\forall\,n=1,2,\dots,\;\;
\|f*(\delta_e-\mu)^\alpha *\mu^{(n)}\|_2\le C_{{\eps} ,\alpha} n^{-\alpha}\|f\|_2.
$$
\end{lem}

\begin{proof} This is a simple consequence of spectral theory and  calculus.  Indeed, (using the spectral resolution $E_\lambda$ of the
(self-adjoint) operator of convolution by $\mu$), spectral theory shows that
$\|f*(\delta_e-\mu)^\alpha *\mu^{(n)}\|_2\le M \|f\|_2$,  where
$$M=\sup_{\lambda\in J_{\eps} }\{  |1-\lambda|^\alpha |\lambda|^n\},\;\;J_{\eps} =[-1+{\eps} ,1].$$
The  local   maxima of the function $\lambda\mapsto |1-\lambda|^\alpha|\lambda|^n$ on $J_{\eps} $ are at $ -1+{\eps} $
 or   $n/(\alpha+n)$
so that
 $$M\le \max\left\{ (1-{\eps} )^n(2-{\eps} )^\alpha, \left(\frac{n}{\alpha+n}\right)^n \left(\frac{\alpha}{\alpha+n}\right)^\alpha\right\} \le
C_{{\eps} ,\alpha}  n^{-\alpha} .$$
\end{proof}
\begin{lem} Fix ${\eps}  \in (0,1]$. If $\mu$ is symmetric and the spectrum of $f\mapsto f*\mu$ on $L^2(G)$ is contained in $[-1+{\eps} ,1]$, then  there exists a constant $C_{{\eps} }$ such that, for all pairs of positive integers $u,v$ such that $n\ge u+2v$, we have
$$\|\mu^{(n+m)}-\mu^{(n)}\|_\infty \le  C_{\eps}  \frac{m}{u}  \mu^{(2v)}(e).$$
\end{lem}
\begin{proof} It suffices to prove this with $m=1$. To that end, observe that
\begin{align*}\|\mu^{(n+1)}-\mu^{(n)}\|_\infty&\le \|\mu^{(u+v+1)}-\mu^{(u+v)}\|_2\|\mu^{(v)}\|_2\\
& = \|\mu^{(v)}*(\mu^{(u+1)}-\mu^{(u)})\|_2 \|\mu^{(v)}\|_2\\
&\le  C_{{\eps} ,1} u^{-1} \|\mu^{(v)}\|_2^2.
\end{align*}
The result follows because $\|\mu^{(v)}\|_2^2=\mu^{(2v)}(e)$.
\end{proof}
\begin{pro} \label{pro-93}
Fix ${\eps}  \in (0,1]$. Let $G$ be a discrete group equipped with a symmetric probability
measure $\mu$, a quasi-norm $\|\cdot\|$ and a continuous
increasing function $r : (0,\infty)\rightarrow (0,\infty)$ with inverse $\rho$. Assume  $\mu$ satisfies a pointwise $(\|\cdot\|,r)$-pseudo-Poincar\'e inequality  with constant $C$ and that the spectrum of $f\mapsto f*\mu$ on $L^2(G)$ is contained in $[-1+{\eps} ,1]$.  Then there exists a constant $C_{\eps} $ such that for all positive integers $n,m,u,v$ such that $n=u+2v$ and all $x,y\in G$, we have
 $$ |\mu^{(n+m)}(xy)-\mu^{(n)}(x)|\le C_{\eps} \left(\frac{m}{u}+ \sqrt{\frac{C\rho(\|y\|)}{u}} \right) \mu^{(2v)}(e).$$
\end{pro}
\begin{proof} It suffices to prove the case $m=0$ (for $m>0$, use the previous lemma to reduce to the case $m=0$). By the Cauchy-Schwarz inequality and the assumed Poincar\'e inequality,
\begin{align*}
|\mu^{(n)}(xy)-\mu^{(n)}(x)| &\le  \|\mu^{(u+v)}(\cdot y)-\mu^{(u+v)}(\cdot)\|_2\|\mu^{(v)}\|_2\\
&\le [C \rho(\|y\|) \mathcal E_\mu(\mu^{(u+v)},\mu^{(u+v)})]^{1/2} \|\mu^{(v)}\|_2\\
 &\le  [C \rho(\|y\|)]^{1/2} \|   \mu^{(v)} *(\delta_e-\mu)^{1/2}*\mu^{(u)}\|_2] \|\mu^{(v)}\|_2\\
 &\le  C_{{\eps} ,1/2}\sqrt{ \frac{C\rho(\|y\|)}{u}} \|\mu^{(v)}\|_2^2\\
 &= C_{{\eps} ,1/2}\sqrt{ \frac{C\rho(\|y\|)}{u}} \mu^{(2v)}(e).
  \end{align*}
\end{proof}
\ \

\noindent \textbf{Acknowledgements.}
We are grateful to the referee for helpful comments. 
The research of ZC is supported by Simons Foundation Grant 520542,
a Victor Klee Faculty Fellowship at UW, and NNSFC grant 11731009,
TK by the Grant-in-Aid for Scientific Research (A) 17H01093,
Japan, LSC by NSF grants DMS-1404435 and DMS-1707589, and JW by NNSFC grant 11831014, the Program for Probability and Statistics: Theory and Application (No.\ IRTL1704), and the Program for Innovative Research Team in Science and Technology in Fujian Province University (IRTSTFJ).

\bigskip
\vskip 0.2truein

\footnotesize{
 {\bf Zhen-Qing Chen}

Department of Mathematics, University of Washington, Seattle,
WA 98195, USA.

E-mail: {\tt zqchen@uw.edu}

\bigskip

{\bf Takashi Kumagai}

Research Institute for Mathematical Sciences,
Kyoto University, Kyoto 606-8502, Japan.

E-mail: {\tt kumagai@kurims.kyoto-u.ac.jp}

\bigskip

{\bf Laurent Saloff-Coste}

Department of Mathematics, Cornell University, Ithaca, NY 14853, USA.

E-mail: {\tt lsc@uno.math.cornell.edu}

\bigskip

{\bf Jian Wang}

College of Mathematics and Informatics, \\
\indent Fujian Key Laboratory of Mathematical Analysis and Applications (FJKLMAA),\\  \indent Center for Applied Mathematics of Fujian Province (FJNU),\\
\indent Fujian Normal University, Fuzhou 350007,
P.R. China. E-mail: {\tt jianwang@fjnu.edu.cn}

\bigskip

{\bf Tianyi Zheng}

Department of Mathematics, UC San Diego, San Diego, CA 92093-0112, USA.

E-mail: {\tt tzheng2@math.ucsd.edu}
}
\end{document}